\documentclass[a4paper,11pt]{article}

\title{$L^q$-spectra of self-affine measures: \\ closed forms, counterexamples, and split binomial sums}

\usepackage{multicol}
\usepackage{float}
\usepackage{amsmath}
\usepackage{amsthm}
\usepackage{amsmath}
\usepackage{amsthm}
\usepackage{amscd}
\usepackage{amsfonts}
\usepackage{amssymb}
\usepackage{latexsym}
\usepackage{verbatim}
\usepackage{enumerate}
\usepackage{graphicx}
\usepackage{empheq}
\usepackage[utf8]{inputenc}
\usepackage[T1]{fontenc}
\usepackage{lmodern} 
\usepackage{parskip}
\usepackage{url}
\usepackage{fancyhdr}
\newtheorem{thm}[equation]{Theorem}
\numberwithin{equation}{section} 

\newtheorem{lem}[equation]{Lemma}
\newtheorem{cor}[equation]{Corollary}
\newtheorem*{cor*}{Corollary}

\theoremstyle{definition}
\newtheorem{ques}[equation]{Question}
\newtheorem{defn}[equation]{Definition}

\theoremstyle{remark}

\theoremstyle{remark}

\date{}

\author{Jonathan M. Fraser, Lawrence D. Lee\footnote{Corresponding Author}, Ian D. Morris \&  Han Yu}

\begin{document}
\maketitle
\begin{abstract}

We study $L^q$-spectra of planar self-affine measures generated by diagonal systems with an emphasis on providing closed form expressions.  We answer a question posed by Fraser in 2016 in the negative by proving that a certain natural closed form expression does not generally give the $L^q$-spectrum and, using a similar approach, find counterexamples to a statement of Falconer-Miao from 2007 and a conjecture of Miao from 2008 concerning a closed form expression for the generalised dimensions of generic self-affine measures.  In the positive direction we provide new non-trivial closed form bounds in both of the above settings, which in certain cases yield sharp results.  We also provide  examples of self-affine measures whose $L^q$-spectra exhibit new types of phase transitions.  Our examples depend on a combinatorial estimate for the exponential growth of certain split binomial sums.


\end{abstract}


\section{Introduction and summary of results}

The $L^q$-spectrum is an important concept in multifractal analysis and quantifies global fluctuations in a given measure.  In the setting of self-affine measures, the $L^q$-spectrum is notoriously difficult to compute, and is only known in some specific cases, see for example \cite{Feng, Fraser} and in some settings a generic formula is known \cite{Barral, Falconerlq1, Falconerlq2}.  Even in some cases where a formula is known, it is not given by a closed form expression which makes explicit calculations (and theoretical manipulation) difficult.  Some attention has been paid to the provision of closed form expressions in \cite{Miao, Fraser,Miao1} and these works provide the main motivation for this one.

First we consider the setting of Fraser \cite{Fraser} and Feng-Wang \cite{Feng}, where the self-affine measures are generated by diagonal systems.  Fraser \cite[Theorem 2.10]{Fraser} provided closed form expressions for the $L^q$-spectra in many cases, but often required some extra assumptions on the defining system.  He asked if these technical assumptions could be removed and if his formula held in general \cite[Question 2.14]{Fraser}.  We answer this question in the negative by providing an explicit family of counterexamples, see Theorem \ref{counterexample}.  Despite the fact that the predicted closed form expression does not hold, we are able to provide new, non-trivial, closed form bounds for the $L^q$-spectra, see Theorem \ref{Lower Bound Theorem}.   We also provide  examples of self-affine measures whose $L^q$-spectra exhibit new types of phase transitions, see Theorem \ref{phasethm}.  Specifically, we construct examples where the $L^q$-spectrum is differentiable at $q=1$ but not analytic in any neighbourhood of $q=1$.

Secondly, we consider the setting of Falconer-Miao \cite{Miao} and Miao \cite{Miao1} where the self-affine measures are generated by upper triangular matrices.   The paper \cite{Miao} was mainly concerned with dimensions of self-affine sets, but towards the end it states a closed form expression for the generalised $q$-dimensions (these are a normalised version of the $L^q$-spectra) in a natural generic setting \cite[Theorem 4.1]{Miao}.  The proof of this result was just sketched and when the result appeared later in Miao's thesis \cite[Theorem 3.11]{Miao1} the full proof was only given for $0<q<1$ and the formula only conjectured to hold for $q>1$.  We show that this formula and conjecture of Miao are false for $q>1$ in general by  providing an explicit family of counterexamples, see Theorem \ref{Miao Counterexample}. We are able to provide new, non-trivial, closed form bounds for the generalised $q$-dimensions, see Theorem \ref{Upper Bound Theorem} and also give new conditions which guarantee that the conjectured formula does hold, see Corollary \ref{Equality}.

A key technical tool is the following growth result for split binomial sums: if one considers the binomial expansion of $(1+x)^k$, where $x>1$ is fixed, and splits the sum in half, then the ratio of the two halves grows exponentially in $k$, see  Theorem \ref{binomial}.

\section{Preliminaries and split binomial sums}

 For background on iterated function systems (IFS) see \cite{Falconer}. We recall some basic definitions.\\

\begin{defn}[Self-affine set]
Suppose we have an IFS $\lbrace S_{i}\rbrace_{i\in\mathcal{I}}$ consisting of contracting affine transformations of $\mathbb{R}^n$ where $\mathcal{I}$ is some finite index set. Then there is a unique non-empty, compact set $F$ satisfying
\begin{equation*}
F=\bigcup_{i\in\mathcal{I}}S_{i}(F)
\end{equation*}
which we call the \textit{self-affine set} associated to $\lbrace S_{i}\rbrace_{i\in\mathcal{I}}$.
\end{defn}

\vspace{5mm}

We are interested in measures on such sets. A natural type of measure on self-affine sets one can construct is a \textit{self-affine} measure.

\vspace{5mm}

\begin{defn}[Self-affine measure]
Suppose we have a self-affine set $F$ given by the IFS $\lbrace S_{i}\rbrace_{i\in\mathcal{I}}$ acting on $\mathbb{R}^n$, and a probability vector $\lbrace p_{i}\rbrace_{i\in\mathcal{I}}$ with each $p_{i}\in (0, 1)$. Then there is a unique Borel probability measure $\mu$ on $\mathbb{R}^n$ satisfying
\begin{equation*}
\mu=\sum_{i\in \mathcal{I}}p_{i}\ \mu  \circ  S_{i}^{-1}
\end{equation*}
which we call the \textit{self-affine measure} associated to $\lbrace S_{i}\rbrace_{i\in\mathcal{I}}$ and  $\lbrace p_{i}\rbrace_{i\in\mathcal{I}}$.
\end{defn}

We close this section with a technical result which states that a certain split binomial sum ratio  grows exponentially.   This result will be used to provide counterexamples later in the paper.\\

\begin{thm}\label{binomial}
Let $x>1$, then
\[\lim_{k \to \infty} \left(\frac{\sum_{i=\lceil k/2\rceil}^k {k \choose i}x^i}{\sum_{i=0}^{\lfloor k/2\rfloor} {k \choose i} x^i}\right)^{\frac{1}{k}}=\frac{1+x}{2\sqrt{x}}>1\]
where the limit is taken along odd integers $k$.
\end{thm}
\begin{proof}
Fix $x>1$ and let $k \geq 1$ be odd. Since ${k \choose i} \leq {k \choose \lfloor k/2\rfloor}$ for all $i=0,\ldots,k$ we have $ {k \choose \lfloor k/2\rfloor} \geq \frac{1}{k+1}\sum_{i=0}^k {k \choose i} =\frac{2^k}{k+1}$. Hence
\[\frac{2^kx^{\lfloor k/2\rfloor}}{k+1} \leq {k \choose \lfloor k/2\rfloor} x^{\lfloor k/2 \rfloor} \leq \sum_{i=0}^{\lfloor k/2\rfloor} {k \choose i}x^i \leq \sum_{i=0}^k {k \choose i}x^{\lfloor k/2\rfloor}=2^kx^{\lfloor k/2\rfloor}.\]
It follows that on the one hand
\[\frac{\sum_{i=\lceil k/2\rceil}^k {k \choose i}x^i}{\sum_{i=0}^{\lfloor k/2\rfloor} {k \choose i} x^i}= \frac{\sum_{i=0}^k {k \choose i}x^i -  \sum_{i=0}^{\lfloor k/2\rfloor} {k \choose i} x^i}{\sum_{i=0}^{\lfloor k/2\rfloor} {k \choose i} x^i}=\frac{(1+x)^k}{\sum_{i=0}^{\lfloor k/2\rfloor}{k \choose i}} -1\geq \frac{(1+x)^k}{2^k x^{\lfloor k/2\rfloor}} -1\]
and on the other hand
\[\frac{\sum_{i=\lceil k/2\rceil}^k {k \choose i}x^i}{\sum_{i=0}^{\lfloor k/2\rfloor} {k \choose i} x^i}\leq \frac{\sum_{i=0}^k {k \choose i}x^i}{\sum_{i=0}^{\lfloor k/2\rfloor} {k \choose i}x^i}=\frac{(1+x)^k}{\sum_{i=0}^{\lfloor k/2\rfloor}{k \choose i}x^i} \leq \frac{(k+1)(1+x)^k}{2^k x^{\lfloor k/2\rfloor}}.\]
Since $\frac{1+x}{2\sqrt{x}}>1$ by the arithmetic-geometric mean inequality the result follows easily.
\end{proof}

\section{Diagonal systems and the $L^q$-spectrum}

We now turn to the first class of IFS we shall study and introduce the $L^q$-spectrum of the associated self-affine measure. We begin by introducing the necessary background from \cite{Fraser, Fraser1}.


\vspace{5mm}

\begin{defn}[$L^q$-spectrum] If $\mu$ is a Borel probability measure on $\mathbb{R}^n$ with support denoted by $\textnormal{supp}(\mu)$ then the upper and lower $L^q$-spectrum of $\mu$ are defined to be 
\begin{equation*}
\overline{\tau}_{\mu}(q)= \overline{\lim}_{\delta\rightarrow 0}\frac{\log\int_{\textnormal{supp}(\mu)}\ \mu(B(x, \delta))^{q-1} \ d\mu(x)}{-\log\delta}
\end{equation*}
and 
\begin{equation*}
\underline{\tau}_{\mu}(q)= \underline{\lim}_{\delta\rightarrow 0}\frac{\log\int_{\textnormal{supp}(\mu)}\ \mu(B(x, \delta))^{q-1} \ d\mu(x)}{-\log\delta}.
\end{equation*}
respectively. If these two values coincide we define the $L^q$-spectrum of $\mu$, denoted $\tau_{\mu}(q)$, to be the common value.
\end{defn}

This quantity is of special interest in multifractal analysis due to its relationship with the fine multifractal spectrum. In particular if the multifractal formalism holds then the fine multifractal spectrum of $\mu$ is given by the Legendre transform of $\tau_{\mu}$ (for details see \cite{Olsen}).


\vspace{5mm}

\begin{defn}[Diagonal System]
We say a self-affine IFS is a \textit{diagonal system} if it is an IFS consisting of affine transformations of $\mathbb{R}^2$ whose linear part is a contracting diagonal matrix. 
\end{defn}

Note that necessarily the maps that make up diagonal systems are of the form $S_{i}=T_{i}+t_{i}$, where $T_{i}$ is a contracting linear map of the form 
\begin{equation*}
T_{i}=\begin{pmatrix}
  \pm c_{i} & 0  \\
  0 &  \pm d_{i} 
 \end{pmatrix}
\end{equation*} 
with $c_{i}, d_{i}\in (0,1)$ and $t_{i}\in\mathbb{R}^2$ is a translation vector. 







We shall also assume that our IFS satisfies the following separation condition.

\vspace{5mm}

\begin{defn}[Rectangular Open Set Condition]
We say an IFS acting on $\mathbb{R}^2$ satisfies the \textit{Rectangular Open Set Condition} (ROSC) if there exists a non-empty open rectangle $R=(a, b)\times(c, d)\subset \mathbb{R}^2$ such that $\lbrace S_{i}(R)\rbrace_{i\in\mathcal{I}}$ are pairwise disjoint subsets of $R$.
\end{defn}

In order to calculate the $L^q$-spectrum $\tau_{\mu}(q)$ of such measures, Fraser introduced what he termed a \textit{q-modified singular value function}. To introduce this we begin by defining the projection maps $\pi_{1},\pi_{2}:\mathbb{R}^2\rightarrow\mathbb{R}$ by $\pi_{1}(x, y)=x$ and $\pi_{2}(x, y)=y$. It may be shown that the projections of the measure $\mu$, namely  $\pi_{1}(\mu)$ and $\pi_{2}(\mu)$, are a pair of self-similar measures. Therefore, it follows from a result of Peres and Solomyak \cite{Peres} that the $L^q$-spectra of both of these projected measures, which we denote by $\tau_{1}(q):=\tau_{\pi_{1}(\mu)}(q)$ and $\tau_{2}(q):=\tau_{\pi_{2}(\mu)}(q)$, exist for $q\geq 0$.   

Let $\mathcal{I}^*=\bigcup_{k\geq 1}\mathcal{I}^k$ denote the set of all finite sequences with entries in $\mathcal{I}$. For $\boldsymbol{i}=(i_{1},\dots, i_{k})\in\mathcal{I}$ let $S_{\boldsymbol{i}}= S_{i_{1}}\circ S_{i_{2}}\circ \cdots\circ S_{i_{k}}$ and let $p(\boldsymbol{i})= p_{i_{1}}p_{i_{2}}\cdots p_{i_{k}}$. Also write $\alpha_{1}(\boldsymbol{i})\geq \alpha_{2}(\boldsymbol{i})$ for the singular values of the linear part of $S_{\boldsymbol{i}}$ and write $c({\boldsymbol{i}})=c_{i_{1}}c_{i_{2}}\cdots c_{i_{k}}$ and $d({\boldsymbol{i}})=d_{i_{1}}d_{i_{2}}\cdots d_{i_{k}}$.  In particular, for all $\boldsymbol{i}=(i_{1},\dots, i_{k})\in\mathcal{I}^*$,  $\alpha_{1}(\boldsymbol{i}) = \max \{c({\boldsymbol{i}}), d({\boldsymbol{i}})\}$ and  $\alpha_{2}(\boldsymbol{i}) = \min \{c({\boldsymbol{i}}), d({\boldsymbol{i}})\}$.


Now define $\pi_{\boldsymbol{i}}:\mathbb{R}^2\rightarrow \mathbb{R}$ by
\begin{equation*}
\pi_{\boldsymbol{i}}=\left\{
        \begin{array}{ll}
            \pi_{1} & \quad \textnormal{if}\  c(\boldsymbol{i})\geq d(\boldsymbol{i}) \\
             \pi_{2} & \quad  \textnormal{if} \ d(\boldsymbol{i})< c(\boldsymbol{i})\\
        
        \end{array}
    \right.
\end{equation*}
and subsequently define $\tau_{\boldsymbol{i}}(q)$ by $\tau_{\boldsymbol{i}}(q):=\tau_{\pi_{\boldsymbol{i}}(\mu)}(q)$. Note that $\tau_{\boldsymbol{i}}(q)$ is simply the $L^q$-spectrum of the projection of $\mu\vert_{S_{\boldsymbol{i}}(F)}$ onto the longest side of the rectangle $S_{\boldsymbol{i}}([0, 1]^2)$ and is always equal to either $\tau_1(q)$ or $\tau_2(q)$.

For $s\in\mathbb{R}$ and $q\geq 0$, define the \textit{q-modified singular value function}, $\psi^{s, q}:\mathcal{I}^*\rightarrow (0,\infty)$  by
\begin{equation*}
\psi^{s, q}(\boldsymbol{i}) = p(\boldsymbol{i})^q\alpha_{1}(\boldsymbol{i})^{\tau_{\boldsymbol{i}}(q)}\alpha_{2}(\boldsymbol{i})^{s-\tau_{\boldsymbol{i}}(q)}
\end{equation*}
and for each $k\in\mathbb{N}$ define the value $\Psi_{k}^{s, q}$ by
\begin{equation}\label{Psi}
\Psi_{k}^{s, q}=\sum_{\boldsymbol{i}\in\mathcal{I}^k}\psi^{s, q}(\boldsymbol{i}).
\end{equation}
It now follows from Lemma 2.2 in \cite{Fraser} and standard properties of sub-multiplicative sequences that we may define a function $P:\mathbb{R}\times [0, \infty)\rightarrow [0,\infty)$ by
\begin{equation*}
P(s, q)= \lim_{k\rightarrow\infty}(\Psi_{k}^{s, q})^{1/k}.
\end{equation*}

It follows from Lemma 2.3 in \cite{Fraser} that we may define another function, $\gamma:[0, \infty)\rightarrow\mathbb{R}$, by $P(\gamma(q), q)=1$. We shall refer to this function as a \textit{moment scaling function}. The importance of this function is the following theorem from \cite{Fraser}.\\

\begin{thm}$\cite[\textnormal{ Theorem 2.6}]{Fraser}$
Suppose that $\mu$  is generated by a diagonal system and satisfies the ROSC. Then
\begin{equation*}
\tau_{\mu}(q)=\gamma(q).
\end{equation*}
\end{thm}

This tells us that finding a closed form expression for  $\tau_{\mu}(q)$  is equivalent to finding a closed form expression form $\gamma(q)$. 

Note that we may approximate $\gamma(q)$ numerically by functions $\gamma_{k}(q)$, where for each $k\in\mathbb{N}$ we define $\gamma_{k}(q):[0, \infty)\rightarrow\mathbb{R}$ by
\begin{equation*}
\Psi^{\gamma_{k}(q), q}_{k}=1.
\end{equation*}
In order to find a closed form expression Fraser defined functions $\gamma_{A}, \gamma_{B}:[0, \infty)\rightarrow\mathbb{R}$ by 
\begin{equation*}
\sum_{i\in\mathcal{I}} p_i^q \ c_i^{\tau_{1}(q)}\ d_i^{\gamma_{A}(q)-\tau_{1}(q)} = 1
\end{equation*}
and
\begin{equation*}
\sum_{i\in\mathcal{I}} p_i^q\ d_i^{\tau_{2}(q)}\ c_i^{\gamma_{B}(q)-\tau_{2}(q)} =1.
\end{equation*}
The following lemma tells us some useful information about the relationship between $\gamma_{A}, \gamma_{B}$ and $\tau_{1},\tau_{2}$.\\

\begin{lem}$\cite[\textnormal{ Lemma 2.9}]{Fraser}$
Let $q\geq 0$. Then either
\begin{equation*}
\max\lbrace\gamma_{A}(q), \gamma_{B}(q)\rbrace\leq\tau_{1}(q)+\tau_{2}(q)
\end{equation*}
or
\begin{equation*}
\min\lbrace\gamma_{A}(q), \gamma_{B}(q)\rbrace\geq\tau_{1}(q)+\tau_{2}(q).
\end{equation*}
\end{lem}

This lemma is particularly helpful as it allows us to state Fraser's main result on closed form expressions  from \cite{Fraser}.  \\

\begin{thm}$\cite[\textnormal{ Theorem 2.10}]{Fraser}$\label{Fraser1}
Let $\mu$ be generated by a diagonal system and $q\geq 0$. 

If $\max\lbrace\gamma_{A}(q), \gamma_{B}(q)\rbrace\leq\tau_{1}(q)+\tau_{2}(q)$ then
\begin{equation*}
\gamma(q)=\max\lbrace\gamma_{A}(q), \gamma_{B}(q)\rbrace.
\end{equation*}
If $\min\lbrace\gamma_{A}(q), \gamma_{B}(q)\rbrace\geq\tau_{1}(q)+\tau_{2}(q)$, then
\begin{equation*}
\tau_{1}(q)+\tau_{2}(q) \leq \gamma(q)\leq\min\lbrace\gamma_{A}(q), \gamma_{B}(q)\rbrace
\end{equation*}
and  if either
\begin{equation*}
\sum_{i\in\mathcal{I}}p_{i}^q \ c_{i}^{\tau_{1}(q)} \ d_{i}^{\gamma_{A}(q)-\tau_{1}(q)}\ \log(c_{i}/d_{i})\geq 0
\end{equation*}
or
\begin{equation*}
\sum_{i\in\mathcal{I}}p_{i}^q \ d_{i}^{\tau_{2}(q)} \ c_{i}^{\gamma_{B}(q)-\tau_{2}(q)}\ \log(d_{i}/c_{i})\geq 0
\end{equation*}
then  $\gamma(q)=\min\lbrace\gamma_{A}(q), \gamma_{B}(q)\rbrace$.
\end{thm}

The fact that we only have an inequality involving $\gamma(q)$ when $\min\lbrace\gamma_{A}(q), \gamma_{B}(q)\rbrace\geq\tau_{1}(q)+\tau_{2}(q)$,  combined with the observation that the above conditions (the sums involving logarithms) do not look especially natural, led Fraser to ask the following question.\\

\begin{ques}$\cite[\textnormal{ Question 2.14}]{Fraser}$ \label{fraserquestion}

If $\min\lbrace\gamma_{A}(q), \gamma_{B}(q)\rbrace\geq\tau_{1}(q)+\tau_{2}(q)$ and neither 
\begin{equation*}
\sum_{i\in\mathcal{I}}p_{i}^q \ c_{i}^{\tau_{1}(q)} \ d_{i}^{\gamma_{A}(q)-\tau_{1}(q)}\ \log(c_{i}/d_{i})\geq 0
\end{equation*}
nor
\begin{equation*}
\sum_{i\in\mathcal{I}}p_{i}^q \ d_{i}^{\tau_{2}(q)} \ c_{i}^{\gamma_{B}(q)-\tau_{2}(q)}\ \log(d_{i}/c_{i})\geq 0
\end{equation*}
are satisfied, is it still true that
\begin{equation*}
\gamma(q)=\min\lbrace\gamma_{A}(q), \gamma_{B}(q)\rbrace?
\end{equation*}
\end{ques}

By presenting a family of counterexamples we shall answer this question in the negative. In particular we  provide a family of diagonal systems consisting of two maps equipped with the Bernoulli-(1/2, 1/2) measure such that 
\[
\gamma(q)<\min\lbrace\gamma_{A}(q), \gamma_{B}(q)\rbrace
\]
for all $q> 1$.

\subsection{A family of counterexamples}

We now turn our attention to the provision of  examples answering Question \ref{fraserquestion} in the negative. We require a family of measures such that the two conditions in Theorem \ref{Fraser1} fail. At the same time we also need to ensure that they are simple enough to allow us to estimate $\Psi_{k}^{s, q}$ (\ref{Psi}) effectively. We prove the following result, which states that, for a certain explicit family of self-affine measures generated by diagonal systems, $\tau_{\mu}(q)$ is not equal to either  $\gamma_A(q)$ or $\gamma_B(q)$ for all $q>1$. Theorem \ref{binomial} will be of key importance  in establishing this result. \\

\begin{thm}\label{counterexample}
Let $c, d$ be such that $c>d>0$ and $c+d\leq 1$. Let  $\mu$ be the self-affine measure defined by the probability vector $(1/2, 1/2)$ and the diagonal system consisting of the two maps, $S_{1}$ and $S_{2}$, where
\[
S_{1}(x, y)=\begin{pmatrix}
  c & 0  \\
  0 &  d
 
 \end{pmatrix}\begin{pmatrix}
  x  \\
  y 
 
 \end{pmatrix}
\qquad \text{ and  } \qquad 
S_{2}(x, y)=\begin{pmatrix}
  d & 0  \\
  0 &  c
 
 \end{pmatrix}\begin{pmatrix}
  x  \\
  y 
 
 \end{pmatrix}+\begin{pmatrix}
  1-d  \\
  1-c 
 
 \end{pmatrix}.
\]
Then, for  $q>1$,
\[
\gamma(q)<\min\lbrace\gamma_{A}(q), \gamma_{B}(q)\rbrace.
\]
More precisely, for $q>1$, $\gamma_A(q) = \gamma_B(q)<0$ and, writing $s$ to denote this common value,
\begin{equation}\label{Quantitative}
\gamma(q) \leq s - \frac{2 \log \left(\frac{2 (d/c)^{s/2}}{(d/c)^s+1}\right)}{\log(cd)}.
\end{equation}
\end{thm}

\begin{proof}

Let $q>1$. We begin by noting that due to the relative simplicity of the maps we are working with it is straightforward to show that $\tau_1(q)=\tau_2(q)=\gamma_A(q)=\gamma_B(q)$. We shall denote this common value by $s$, and also note that $s<0$.

Let $k$ be odd. We may write $\Psi_k^{s,q}$ as 
\begin{equation}\label{Psi Calculation}
\begin{split}
\Psi_k^{s,q}&=\sum_{\boldsymbol{i}\in\mathcal{I}^k}p_{\boldsymbol{i}}^q\ \alpha_1(\boldsymbol{i})^{\tau_{\boldsymbol{i}}(q)} \ \alpha_2(\boldsymbol{i})^{s-\tau_{\boldsymbol{i}}(q)} \\
 & = \sum_{\boldsymbol{i}\in\mathcal{I}^k}2^{-kq}\ \alpha_1(\boldsymbol{i})^s,\\
 \end{split}
\end{equation}
using the fact that $p=1/2$ and $s=\tau_1(q)=\tau_2(q)$. Since the maps $S_1$ and $S_2$ commute, we can write each $S_{\boldsymbol{i}}$ ($\boldsymbol{i}\in\mathcal{I}^k$) as $S_{\boldsymbol{i}} = S_1^i \circ S_2^{k-i}$ where $i \in [0,k]$ is the number of times $S_1$ was used in the composition of $S_{\boldsymbol{i}}$.  For such maps, since $c>d$,
\[
\alpha_1(\boldsymbol{i})=c^{\max\{i, k-i\}}\times d^{\min\{i, k-i\}}
\]
and we can re-express (\ref{Psi Calculation}) as
\[
\Psi_k^{s,q}= X_k^{q}+Y_k^{q},
\]
where
\[
X_k^{q}=\sum_{i=0}^{\lfloor k/2 \rfloor}\binom{k}{i}2^{-kq}\left(c^{k-i}d^{i}\right)^s
\]
and 
\[
Y_k^{q}=\sum_{i=\lceil k/2 \rceil}^{k}\binom{k}{i}2^{-kq}\left(d^{k-i}c^{i}\right)^s.
\]
We now consider the ratio $X_k^{q}/(1-X_k^{q})$. By our binomial result (Theorem \ref{binomial}) and the definition of $s=\gamma_{A}(q)$,
\[
\sum_{i=0}^{k}\binom{k}{i}2^{-kq}\left(c^{k-i}d^{i}\right)^s=\left(2^{-q}c^{\gamma_{A}(q)} +2^{-q}d^{\gamma_{A}(q)} \right)^k=1^k=1
\]
and therefore 
\[
\frac{X_k^{q}}{1-X_k^{q}}=\frac{\sum_{i=0}^{\lfloor k/2 \rfloor}\binom{k}{i}2^{-kq}\left(c^{k-i}d^{i}\right)^s}{\sum_{i=\lceil k/2 \rceil}^{k}\binom{k}{i}2^{-kq}\left(c^{k-i}d^{i}\right)^s}.
\]
We may rearrange (and cancel a factor $2^{-kq}c^{ks}$) to give
\[
\frac{X_k^{q}}{1-X_k^{q}}=\frac{\sum_{i=0}^{\lfloor k/2 \rfloor}\binom{k}{i}\left(\left(d/c\right)^s\right)^i}{\sum_{i=\lceil k/2 \rceil}^{k}\binom{k}{i}(\left(d/c\right)^s)^i}.
\]
We note that as $c>d$ and as $s<0$ we have $(d/c)^s>1$. Thus by Theorem \ref{binomial}, 
\[
\left(\frac{X_k^{q}}{1-X_k^{q}} \right)^{1/k} \to \frac{2 (d/c)^{s/2}}{(d/c)^s+1} =: \delta \in (0,1)
\]
as $k \to \infty$. Thus we also have  $\left(X_k^{q}\right)^{1/k} \to \delta$ as $k \to \infty$.  By following similar reasoning  we can deduce the same result for $Y_k^{q}$. In particular,
\begin{equation}\label{Binomial Calculation}
\frac{Y_k^{q}}{1-Y_k^{q}}=\frac{\sum_{i=\lceil k/2 \rceil}^{k}\binom{k}{i}2^{-kq}\left(d^{k-i}c^{i}\right)^s}{\sum_{i=0}^{\lfloor k/2 \rfloor}\binom{k}{i}2^{-kq}\left(d^{k-i}c^{i}\right)^s}=\frac{\sum_{i=\lceil k/2 \rceil}^{k}\binom{k}{i}\left(d^{k-i}c^{i}\right)^s}{\sum_{i=0}^{\lfloor k/2 \rfloor}\binom{k}{i}\left(d^{k-i}c^{i}\right)^s}
\end{equation}

which equals
\begin{equation}\label{Binomial Calculation}
 \frac{\sum_{j=0}^{\lfloor k/2 \rfloor}\binom{k}{j}\left((d/c)^s \right)^j}{\sum_{j=\lceil k/2 \rceil}^{k}\binom{k}{j}\left((d/c)^s\right)^j}
\end{equation}
(this follows from relabelling the summation by $j=k-i$ and using the fact that $\binom{k}{k-j}=\binom{k}{j}$).
Note that (\ref{Binomial Calculation}) gives exactly the same as the expression we found for $X_k^{q}/(1-X_k^{q})$ earlier, and so we must also have $\left(Y_k^{q}\right)^{1/k} \to \delta$ as $k \to \infty$.  Therefore
\[
P(s, q)  = \lim_{k\rightarrow\infty}\left(\Psi_k^{s,q} \right)^{1/k}= \lim_{k\rightarrow\infty}\left(X_k^{q} + Y_k^{q} \right)^{1/k}=\delta<1
\]
and by definition of $P(t,q)$ and $\gamma(q)$
\[
 P(\gamma(q), q) = 1> \delta =  P(s, q) .
\]
Since $P(t,q)$ is decreasing in $t$ $\gamma(q)<s=\gamma_A(q)=\gamma_B(q)$, which is enough to show that $\gamma(q)<\min\{\gamma_A(q),\gamma_B(q)\}$.  We can upgrade this result to get the stated quantitative upper bound (\ref{Quantitative}) by considering the function $P(t,q)$ more closely.   For  $k \geq 1$ and $\boldsymbol{i}\in\mathcal{I}^k$,  $\alpha_1(\boldsymbol{i}) \geq (cd)^{ k/2  }$ and therefore, for $\varepsilon = s-\gamma(q)>0$,
\begin{eqnarray*}
\delta = P(s, q)= \lim_{k\rightarrow\infty}\left(\sum_{\boldsymbol{i}\in\mathcal{I}^k}2^{-kq}\ \alpha_1(\boldsymbol{i})^{\gamma(q)+\varepsilon}\right)^{1/k} &\geq& \lim_{k\rightarrow\infty}\left((cd)^{\varepsilon  k/2  }\sum_{\boldsymbol{i}\in\mathcal{I}^k}2^{-kq}\ \alpha_1(\boldsymbol{i})^{\gamma(q)}\right)^{1/k} \\
&=& (cd)^{\varepsilon /2}P(\gamma(q), q) \\
&=& (cd)^{\varepsilon /2}
\end{eqnarray*}
and therefore
\[
s-\gamma(q) = \varepsilon  \geq  \frac{2\log \delta}{\log(cd)}
\]
which proves the theorem.
\end{proof}

\subsection{New examples of phase transitions}
Here we record a simple consequence of Theorem \ref{counterexample} relating to phase transitions. We say that the $L^q$-spectrum $\tau_{\mu}(q)$ exhibits a \textit{first order phase transition} at a point $t\in\mathbb{R}$ if the derivative of $\tau_{\mu}$ is discontinuous at $t$. Likewise we say $\tau_{\mu}(q)$ exhibits an \textit{$n$th order phase transition} at $t\in\mathbb{R}$ if its derivatives up to the $(n-1)$th order are continuous at $t$ but the $n$th order derivative is discontinuous at this point. 

The differentiability of the $L^q$-spectrum is important and has many interesting consequences. Key among these is the fact that if $\tau_{\mu}^{\prime}(1)$ exists then its absolute value gives the Hausdorff dimension of the measure in question, see \cite{Ngai}. We can use Theorem \ref{counterexample} to provide examples of behaviour relating to higher order phase transitions at $q=1$. We are unaware of any other method for constructing such examples.\\


\begin{thm} \label{phasethm}
There exists a planar self-affine measure $\mu$ defined by an IFS satisfying the rectangular opens set condition (ROSC) such that $\tau_{\mu}$, the $L^q$-spectrum of $\mu$, is differentiable at $q=1$ but not analytic in any neighbourhood of $q=1$.
\end{thm}

\begin{proof}
Consider the planar self-affine measures considered in Theorem \ref{counterexample}. As the functions $\tau_1, \tau_2$ are the $L^q$-spectra of the measures $\pi_1\mu, \pi_2\mu$ and  these measures are self-similar and satisfy the open set condition,  it follows that they are real analytic on $(0,\infty)$, see \cite[Chapter 17]{Falconer}, (in particular, they are differentiable at $q=1$).  We can therefore apply  Theorem 2.12 in \cite{Fraser} and conclude that the function $\gamma(q)$ is differentiable at $q=1$, so that $\tau_{\mu}=\gamma$ is differentiable at $q=1$.

Observe that the function $\gamma_A =\gamma_B$ is also real analytic on $(0,\infty)$, since it inherits analyticity from $\tau_1, \tau_2$  via the  analytic implicit function theorem. We know that $\gamma(q) = \gamma_A(q) =\gamma_B(q)$ for $q \in [0,1]$ but $\gamma(q) < \gamma_A(q) =\gamma_B(q)$ for $q > 1$, see  Theorem \ref{counterexample}.  It follows that $\tau_{\mu}=\gamma$ cannot be analytic on any neighbourhood of $q=1$.
\end{proof}

\begin{ques}
How many derivatives does $\tau_{\mu}=\gamma$ have at $q=1$ for the measures $\mu$ considered in  Theorem \ref{counterexample}?
\end{ques}


\subsection{New closed form lower bounds}\label{Lower Bounds}

We now know that $\gamma(q)$ is not in general given by either the maximum or minimum of $\gamma_A(q)$ and $\gamma_B(q)$.  However, by developing a quantitative version of the argument in \cite{Fraser} used to prove Theorem \ref{Fraser1} we are able to provide  new closed form lower bounds for $\gamma(q)$ for all planar diagonal systems. Given $x\in\mathbb{R}$ we write $x^+$ to denote the maximum of $x$ and $0$.\\

\begin{thm}\label{Lower Bound Theorem}
Let $\mu$ be a self-affine measure generated by a diagonal system and let $q\geq 0$. Then 
\[
\gamma(q)\geq \max \{L_A(q), L_B(q)\}
\]
where
\[
L_A(q) = \gamma_{A}(q)-\left(\Big(\gamma_{A}(q)-\tau_{1}(q)-\tau_{2}(q)\Big)\frac{\sum_{i\in\mathcal{I}} p_{i}^q \ c_{i}^{\tau_{1}(q)} \ d_{i}^{\gamma_{A}(q)-\tau_{1}(q)}\log (c_{i}/d_{i})}{\sum_{i\in\mathcal{I}} p_{i}^q \ c_{i}^{\tau_{1}(q)} \ d_{i}^{\gamma_{A}(q)-\tau_{1}(q)}\log (c_{i})}\right)^+
\]
and 
\[
L_B(q) = \gamma_{B}(q)-\left(\Big(\gamma_{B}(q)-\tau_{1}(q)-\tau_{2}(q)\Big)\frac{\sum_{i\in\mathcal{I}} p_{i}^q \ d_{i}^{\tau_{2}(q)} \ c_{i}^{\gamma_{B}(q)-\tau_{2}(q)}\log (d_{i}/c_{i})}{\sum_{i\in\mathcal{I}} p_{i}^q \ d_{i}^{\tau_{2}(q)} \ c_{i}^{\gamma_{B}(q)-\tau_{2}(q)}\log (d_{i})}\right)^+.
\]
In particular,
\[
\frac{\sum_{i\in\mathcal{I}} p_{i}^q \ c_{i}^{\tau_{1}(q)} \ d_{i}^{\gamma_{A}(q)-\tau_{1}(q)}\log (c_{i}/d_{i})}{\sum_{i\in\mathcal{I}} p_{i}^q \ c_{i}^{\tau_{1}(q)} \ d_{i}^{\gamma_{A}(q)-\tau_{1}(q)}\log (c_{i})}
\]
and
\[
\frac{\sum_{i\in\mathcal{I}} p_{i}^q \ d_{i}^{\tau_{2}(q)} \ c_{i}^{\gamma_{B}(q)-\tau_{2}(q)}\log (d_{i}/c_{i})}{\sum_{i\in\mathcal{I}} p_{i}^q \ d_{i}^{\tau_{2}(q)} \ c_{i}^{\gamma_{B}(q)-\tau_{2}(q)}\log (d_{i})}
\]
are both strictly less than 1, which ensures that this result provides a strictly better bound than $\gamma(q)\geq \tau_1(q)+\tau_2(q)$ in the case when $\gamma(q) \leq \min\{\gamma_A(q), \gamma_B(q)\}$.
\end{thm}

\begin{proof}
We prove that $\gamma(q)\geq L_A(q)$.   The inequality $\gamma(q)\geq L_B(q)$ follows by an analogous argument which we omit.  Let $\lbrace\theta_{i}\rbrace_{i\in\mathcal{I}}$ denote an arbitrary probability vector, and for each $k\in\mathbb{N}$, define a number $n(k)\in\mathbb{N}$ by
\[
n(k)=\sum_{i\in\mathcal{I}}\lfloor{\theta_{i}k}\rfloor.
\end{equation*}
Note that $k-|\mathcal{I}|\leq n(k)\leq k$. We consider the $n(k)$th iteration of $\mathcal{I}$ and define
\[
\mathcal{J}_{k}=\left\lbrace\boldsymbol{j}=(j_{1},\dots,j_{n(k)})\in\mathcal{I}^{n(k)}:\#\lbrace m:j_{m}=i\rbrace=\lfloor\theta_{i}k\rfloor\textnormal{ for each}\ i\in\mathcal{I}\right\rbrace,
\]
noting that
\[
|\mathcal{J}_{k}|=\frac{n(k)!}{\prod_{i\in\mathcal{I}}\lfloor\theta_{i}k\rfloor !}.
\]
We also define numbers $c$, $d$ and $p$ (for which we suppress the dependency on $k$) by 
\[
c=\prod_{i\in\mathcal{I}}c_{i}^{\lfloor\theta_{i}k\rfloor}, \qquad d=\prod_{i\in\mathcal{I}}d_{i}^{\lfloor\theta_{i}k\rfloor}, \qquad p=\prod_{i\in\mathcal{I}}p_{i}^{\lfloor\theta_{i}k\rfloor}.
\]
First assume that $\prod_{i\in\mathcal{I}}c_{i}^{\theta_i}>\prod_{i\in\mathcal{I}}d_{i}^{\theta_i}$. In particular this assumption  implies that $c>d$ for $k$ sufficiently large. Indeed
\[
c=\prod_{i\in\mathcal{I}}c_{i}^{\lfloor\theta_{i}k\rfloor} \geq \left(\prod_{i\in\mathcal{I}}c_{i}^{\theta_{i}} \right)^k
\]
and
\[
d=\prod_{i\in\mathcal{I}}d_{i}^{\lfloor\theta_{i}k\rfloor} \leq \left(\prod_{i\in\mathcal{I}}d_{i}^{\theta_{i}} \right)^k \left(\prod_{i\in\mathcal{I}}d_{i} \right)^{-1}
\]
and therefore $c>d$ for all
\[
k > \frac{-\log \left(\prod_{i\in\mathcal{I}}d_{i}\right)}{\log\left( \left(\prod_{i\in\mathcal{I}}c_{i}^{\theta_i} \right) / \left( \prod_{i\in\mathcal{I}}d_{i}^{\theta_i} \right)\right)}.
\]

Therefore, for all sufficiently large $k$,  $\boldsymbol{i}\in\mathcal{J}_{k}$ and $s\in\mathbb{R}$, 
\begin{equation}\label{q-mod}
\psi^{s, q}(\boldsymbol{i})=p^q\ c^{\tau_{1}(q)}\ d^{s-\tau_{1}(q)}.
\end{equation}

By definition of $c,d$ and $p$ we may write this as 
\[
\psi^{s, q}(\boldsymbol{i})=\prod_{i\in\mathcal{I}}\left(p_{i}^q \ c_{i}^{\tau_{1}(q)}\ d_{i}^{s-\tau_{1}(q)}\right)^{\lfloor\theta_{i}k\rfloor}.
\]
We now introduce a form of Stirling's approximation which states that for $n\in\mathbb{N}$ sufficiently large 
\[
n\log n-n\leq\log n!\leq n\log n -n +\log n.
\]
Using this as well as (\ref{q-mod}) we find that for $k$ sufficiently large
\begin{align*}
\log\left(\Psi_{n(k)}^{s, q}\right) & \geq \log \left(\sum_{\boldsymbol{i}\in\mathcal{J}_{k}}\psi^{s, q}(\boldsymbol{i}) \right) \\
&=\log\left(|\mathcal{J}_k|\prod_{i\in\mathcal{I}}\left(p_{i}^q \ c_{i}^{\tau_{1}(q)}\ d_{i}^{s-\tau_{1}(q)}\right)^{\lfloor\theta_{i}k\rfloor}\right)\\
&=\left(\log n(k)!-\sum_{i\in\mathcal{I}}\log\lfloor\theta_{i}k\rfloor!+\sum_{i\in\mathcal{I}}\lfloor\theta_{i}k\rfloor\log\left(p_{i}^q \ c_{i}^{\tau_{1}(q)}\ d_{i}^{s-\tau_{1}(q)}\right)\right)\\
&\geq \bigg(n(k)\log n(k) -n(k)- \sum_{i\in\mathcal{I}}\lfloor\theta_{i}k\rfloor\log\lfloor\theta_{i}k\rfloor+\sum_{i\in\mathcal{I}}\lfloor\theta_{i}k\rfloor\\ 
&\qquad\qquad-\sum_{i\in\mathcal{I}}\log\lfloor\theta_{i}k\rfloor +\sum_{i\in\mathcal{I}}\lfloor\theta_{i}k\rfloor\log\left(p_{i}^q \ c_{i}^{\tau_{1}(q)}\ d_{i}^{s-\tau_{1}(q)}\right)\bigg)
\end{align*}
where the last line follows from the above version of Stirling's formula. Continuing to bound and introducing and exponent of $1/n(k)$ we get
\begin{align*}
\log\left(\Psi_{n(k)}^{s, q}\right)^{1/n(k)} & \geq\frac{1}{n(k)} \bigg(n(k)\log n(k)- \sum_{i\in\mathcal{I}}\lfloor\theta_{i}k\rfloor\log k- \sum_{i\in\mathcal{I}}\lfloor\theta_{i}k\rfloor\log\theta_{i} \\ 
&\qquad\qquad-\sum_{i\in\mathcal{I}}\log\lfloor\theta_{i}k\rfloor +\sum_{i\in\mathcal{I}}\lfloor\theta_{i}k\rfloor\log\left(p_{i}^q \ c_{i}^{\tau_{1}(q)}\ d_{i}^{s-\tau_{1}(q)}\right)\bigg)\\
&\geq\frac{1}{n(k)} \Bigg(n(k)\log n(k)-n(k)\log k-\sum_{i\in\mathcal{I}}\log\lfloor\theta_{i}k\rfloor\\
&\qquad\qquad\qquad\qquad +\sum_{i\in\mathcal{I}}\lfloor\theta_{i}k\rfloor \log\left( \frac{p_{i}^q \ c_{i}^{\tau_{1}(q)}\ d_{i}^{s-\tau_{1}(q)}}{\theta_{i}}\right)\Bigg)\\
&\geq \log\left(\frac{k-|\mathcal{I}|}{k}\right)-\frac{1}{k-|\mathcal{I}|}\sum_{i\in\mathcal{I}}\log\theta_{i}k\\ 
&\qquad\qquad\qquad\qquad +\sum_{i\in\mathcal{I}}\theta_{i} \log\left( \frac{p_{i}^q \ c_{i}^{\tau_{1}(q)}\ d_{i}^{s-\tau_{1}(q)}}{\theta_{i}}\right)
\end{align*}



where the last line uses the fact that $k-|\mathcal{I}|\leq n(k)$. Taking the limit as $k\rightarrow\infty$ the right hand side tends to
\[
\sum_{i\in\mathcal{I}}\theta_{i} \log\left( \frac{p_{i}^q \ c_{i}^{\tau_{1}(q)}\ d_{i}^{s-\tau_{1}(q)}}{\theta_{i}}\right).
\]
If this is non-negative then
\[
P(s, q)=\lim_{k\rightarrow\infty}\left(\Psi_{n(k)}^{s, q}\right)^{1/n(k)}\geq 1
\]
and therefore $\gamma(q)\geq s$.

Second, assume that $\prod_{i\in\mathcal{I}}c_{i}^{\theta_i}<\prod_{i\in\mathcal{I}}d_{i}^{\theta_i}$.  In this case, a completely analogous argument proves that  if
\[
\sum_{i\in\mathcal{I}}\theta_{i} \log\left( \frac{p_{i}^q \ d_{i}^{\tau_{2}(q)}\ c_{i}^{s-\tau_{2}(q)}}{\theta_{i}}\right)\geq 0 
\]
then $P(s, q)\geq 1$ and so $\gamma(q)\geq s$.

Finally, if $\prod_{i\in\mathcal{I}}c_{i}^{\theta_i}=\prod_{i\in\mathcal{I}}d_{i}^{\theta_i}$ then we cannot guarantee that $c>d$ or $d>c$ for all $k$ sufficiently large. We can however conclude that we must have either $c\geq d$ or $d\geq c$ (or both) for infinitely many $k$, so by choosing an appropriate subsequence we can reduce to one of the above two cases. Since we do not know which case we are in ($c\geq d$ or $d\geq c$) we must require that both of the above summation conditions hold. Putting the above three cases together we have therefore shown that
\[
\begin{split}
\gamma(q)\geq \sup \Bigg\{ s: \textnormal{there exists a probability vector}\ \lbrace\theta_{i}\rbrace_{i\in\mathcal{I}} \ \textnormal{such that either}\\ (1) \ \prod_{i\in\mathcal{I}}c_{i}^{\theta_i}>\prod_{i\in\mathcal{I}}d_{i}^{\theta_i} \ \textnormal{and}\sum_{i\in\mathcal{I}}\theta_{i} \log\left( \frac{p_{i}^q \ c_{i}^{\tau_{1}(q)}\ d_{i}^{s-\tau_{1}(q)}}{\theta_{i}}\right)\geq 0\\ \textnormal{or}\ (2) \ \prod_{i\in\mathcal{I}}c_{i}^{\theta_i}<\prod_{i\in\mathcal{I}}d_{i}^{\theta_i} \ \textnormal{and}\sum_{i\in\mathcal{I}}\theta_{i} \log\left( \frac{p_{i}^q \ d_{i}^{\tau_{2}(q)}\ c_{i}^{s-\tau_{2}(q)}}{\theta_{i}}\right)\geq 0 \\
\textnormal {or}\ (3) \ \prod_{i\in\mathcal{I}}c_{i}^{\theta_i}=\prod_{i\in\mathcal{I}}d_{i}^{\theta_i}\ \textnormal{and both}\sum_{i\in\mathcal{I}}\theta_{i} \log\left( \frac{p_{i}^q \ c_{i}^{\tau_{1}(q)}\ d_{i}^{s-\tau_{1}(q)}}{\theta_{i}}\right)\geq 0\\ \textnormal{and}\sum_{i\in\mathcal{I}}\theta_{i} \log\left( \frac{p_{i}^q \ d_{i}^{\tau_{2}(q)}\ c_{i}^{s-\tau_{2}(q)}}{\theta_{i}}\right)\geq 0&\Bigg\}.
\end{split}
\]
In the above we have the freedom to choose a probability vector $\{\theta_{i}\}_{i\in\mathcal{I}}$. A natural choice here, suggested by considering Lagrange multipliers, is to take
\[
\{\theta_{i}\}_{i\in\mathcal{I}}=\left\{ p_{i}^q \ c_{i}^{\tau_{1}(q)} \ d_{i}^{\gamma_{A}(q)-\tau_{1}(q)}\right\}_{i\in\mathcal{I}}
\]
(note that this is indeed a probability vector by definition of $\gamma_A$). We now let $s=\gamma_{A}(q)-\varepsilon$ for $\varepsilon\geq 0$. We want to see how small we can make $\varepsilon$ (ideally we want $\varepsilon=0$) such that the two conditions hold simultaneously. The first holds trivially, since
\[
\sum_{i\in\mathcal{I}}p_{i}^q \ c_{i}^{\tau_{1}(q)} \ d_{i}^{\gamma_{A}(q)-\tau_{1}(q)} \log\left( \frac{p_{i}^q \ c_{i}^{\tau_{1}(q)}\ d_{i}^{\gamma_{A}(q)-\varepsilon-\tau_{1}(q)}}{p_{i}^q \ c_{i}^{\tau_{1}(q)} \ d_{i}^{\gamma_{A}(q)-\tau_{1}(q)}}\right)=\sum_{i\in\mathcal{I}}p_{i}^q \ c_{i}^{\tau_{1}(q)} \ d_{i}^{\gamma_{A}(q)-\tau_{1}(q)} \log (d_{i}^{-\varepsilon})\geq 0.
\]
For the second to hold, we require
\[
\sum_{i\in\mathcal{I}}p_{i}^q \ c_{i}^{\tau_{1}(q)} \ d_{i}^{\gamma_{A}(q)-\tau_{1}(q)} \log\left( \frac{p_{i}^q \ d_{i}^{\tau_{2}(q)}\ c_{i}^{\gamma_{A}(q)-\varepsilon-\tau_{2}(q)}}{p_{i}^q \ c_{i}^{\tau_{1}(q)} \ d_{i}^{\gamma_{A}(q)-\tau_{1}(q)}}\right)\geq 0.
\]
Rearranging this, we see that this is equivalent to requiring
\begin{equation}\label{Epsilion Condition}
\varepsilon\geq \Big(\gamma_{A}(q)-\tau_{1}(q)-\tau_{2}(q)\Big)\frac{\sum_{i\in\mathcal{I}} p_{i}^q \ c_{i}^{\tau_{1}(q)} \ d_{i}^{\gamma_{A}(q)-\tau_{1}(q)}\log (c_{i}/d_{i})}{\sum_{i\in\mathcal{I}} p_{i}^q \ c_{i}^{\tau_{1}(q)} \ d_{i}^{\gamma_{A}(q)-\tau_{1}(q)}\log (c_{i})}.
\end{equation}
We  note that when Fraser's original condition from Theorem 2.10 in \cite{Fraser} holds, namely if 
\[
\sum_{i\in\mathcal{I}} p_{i}^q \ c_{i}^{\tau_{1}(q)} \ d_{i}^{\gamma_{A}(q)-\tau_{1}(q)}\log (c_{i}/d_{i})\geq 0,
\]
then right hand side of (\ref{Epsilion Condition}) is negative so we may take $\varepsilon=0$. Otherwise we use the bound for $\varepsilon$ given in (\ref{Epsilion Condition}). Putting these two cases together therefore gives us that
\[
\gamma(q)\geq \gamma_{A}(q)-\left(\Big(\gamma_{A}(q)-\tau_{1}(q)-\tau_{2}(q)\Big)\frac{\sum_{i\in\mathcal{I}} p_{i}^q \ c_{i}^{\tau_{1}(q)} \ d_{i}^{\gamma_{A}(q)-\tau_{1}(q)}\log (c_{i}/d_{i})}{\sum_{i\in\mathcal{I}} p_{i}^q \ c_{i}^{\tau_{1}(q)} \ d_{i}^{\gamma_{A}(q)-\tau_{1}(q)}\log (c_{i})}\right)^+. 
\]
Finally we note that
\[
\frac{\sum_{i\in\mathcal{I}} p_{i}^q \ c_{i}^{\tau_{1}(q)} \ d_{i}^{\gamma_{A}(q)-\tau_{1}(q)}\log (c_{i}/d_{i})}{\sum_{i\in\mathcal{I}} p_{i}^q \ c_{i}^{\tau_{1}(q)} \ d_{i}^{\gamma_{A}(q)-\tau_{1}(q)}\log (c_{i})}= 1-\frac{\sum_{i\in\mathcal{I}} p_{i}^q \ c_{i}^{\tau_{1}(q)} \ d_{i}^{\gamma_{A}(q)-\tau_{1}(q)}\log (d_{i})}{\sum_{i\in\mathcal{I}} p_{i}^q \ c_{i}^{\tau_{1}(q)} \ d_{i}^{\gamma_{A}(q)-\tau_{1}(q)}\log (c_{i})}< 1
\]
so our lower bound is indeed an improvement on 
\[
\gamma(q)\geq \tau_{1}(q)+\tau_{2}(q)
\]
in the case when $\gamma(q) \leq \min\{\gamma_A(q), \gamma_B(q)\}$.
\end{proof}

\subsection{An example}

Here we present an example of a diagonal system satisfying the assumptions of Theorem \ref{counterexample} where we take $c=3/4$ and $d=1/4$. In this setting we know from Theorem \ref{counterexample} that $\tau_{\mu}(q)=\gamma(q)$ is not given by the maximum or minimum of $\gamma_A(q)$ and $\gamma_B(q)$ for $q>1$. It is therefore natural to consider bounds for the $L^q$-spectrum.

Let $q>1$. Focusing on upper bounds Theorem \ref{Fraser1} implies that, for $q>1$, $\gamma_A(q) = \gamma_B(q) = \tau_1(q)=\tau_2(q) = s<0 $, where $s$ is the solution of
\[
2^{-q}c^s+2^{-q}d^s = 1,
\]
and
\[
\gamma(q) \leq s - \frac{2 \log \left(\frac{2 (d/c)^{s/2}}{(d/c)^s+1}\right)}{\log(cd)}.
\] 
Concerning lower bounds,    Theorem \ref{Lower Bound Theorem} implies that 
\[
\gamma(q) \geq \max \{L_A(q), L_B(q)\} =s \left(2 - \frac{c^s\log(d) +d^s\log(c) }{c^s\log(c) +d^s\log(d)} \right).
\]
We also note a couple of trivial lower bounds.  Since $\gamma(0) = 1$ (the box dimension of the support of $\mu$), $\gamma(1)=0$, and $\gamma$ is necessarily convex, it follows that  $1-q$ is a   lower bound for $\tau_{\mu}(q)$. We also know that $\tau_1(q)+\tau_2(q)$ is  a lower bound for $\tau_{\mu}(q)$, see a remark following \cite[Question 2.14]{Fraser}.  Figure \ref{frasercasefig} shows  a plot of these bounds for $q\in [1, 20]$. We  see that our new lower bound, $\max \{L_A(q), L_B(q)\}$ is a strict improvement on the lower bound of $1-q$ outside of the the range $(1.7, 9.3)$.

\begin{figure}[H]
  \centering
\includegraphics[scale=0.7]{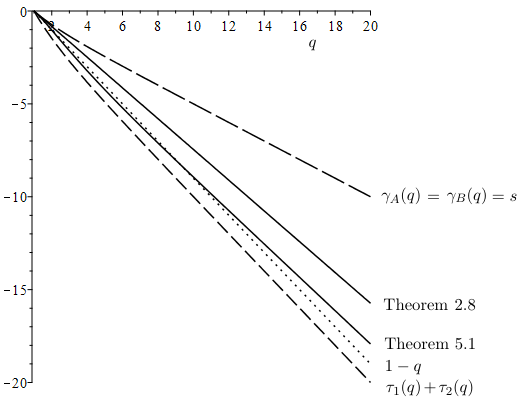}
\caption{Graph of our new upper and lower  bounds for the $L^q$-spectrum (solid lines), labelled by the Theorem they come from.  For reference we also show graphs of the previously known upper bound $\min\lbrace\gamma_{A}(q), \gamma_{B}(q)\rbrace$ (long dash) and the previously known lower bound $\tau_1(q)+\tau_2(q)$ (short dash), as well as the lower bound  $1-q$, which is specific to this setting (dots).}
\label{frasercasefig}
\end{figure}

\begin{figure}[H]
  \centering
\includegraphics[width=0.8\textwidth]{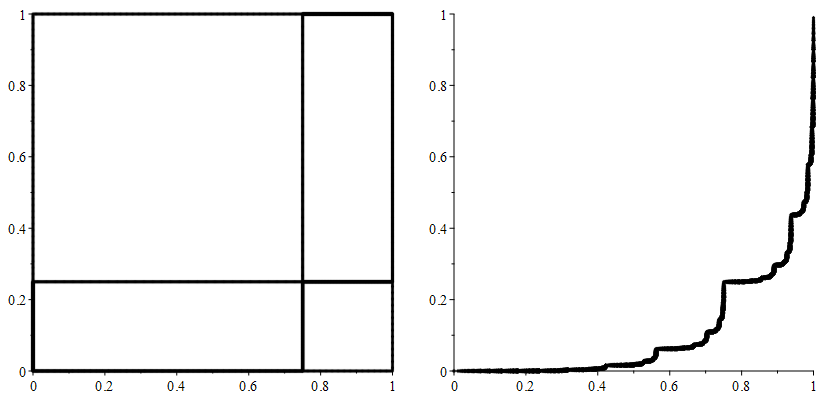}
\caption{Left: images of the unit square under the two maps used above. Right: the associated self-affine set.}
\label{frasercasefig2}
\end{figure}

\section{Generalised $q$-dimensions in the generic setting}

In \cite{Miao} Falconer and Miao considered self-affine sets and measures generated by IFS consisting of upper-triangular matrices.  This paper was mainly concerned with dimensions of self-affine \emph{sets}, but towards the end of the paper they stated a closed form expression for the \textit{generalised $q$-dimensions} in the  measure setting (here, generalised $q$-dimensions simply refer to the $L^q$-spectrum normalised by $1-q$). We show that in fact their formula does \emph{not} always hold when $q>1$. We begin by recalling some definitions and notation from \cite{Miao}.\\

\begin{defn}
Suppose $T$ is an $n\times n$ contracting matrix. Then for $0\leq s\leq n$ the \textit{singular value function} $\phi^s(T)$ is defined to be
\[
\phi^s(T)=\alpha_1\alpha_2\cdots\alpha_{m-1}\alpha_m^{s-m+1}
\]
where $\alpha_1\geq\alpha_2\geq\cdots\geq\alpha_n$ are the singular values of $T$ and where $m$ is the unique integer such that $m-1<s\leq m$. For $s\geq n$ we define $\phi^s(T)$ to be
\[
\phi^s(T)=(\alpha_1\alpha_2\cdots\alpha_{n})^{s/n}.
\]
\end{defn}

For a finite Borel measure $\mu$ on $\mathbb{R}^n$  and $q\in\mathbb{R}$, $q\neq 1$, Falconer and Miao discuss the \textit{generalised q-dimensions} of $\mu$, denoted $D_q(\mu)$. This is simply defined to be the $L^q$-spectra of $\mu$ normalised by $1-q$, that is
\[
D_q(\mu)=\frac{\tau_{\mu}(q)}{1-q}
\]
provided the appropriate limits exist. In order to calculate the generalised $q$-dimensions of self-affine measures $\mu$ associated with contracting upper triangular matrices $T_1,\dots, T_N$ and probabilities $p_1,\dots, p_N$ Falconer and Miao studied the quantity $d_q(T_1,\dots, T_N, \mu)$ defined, for each $q\geq 0$ ($q \neq 1$) to be the unique $t$ satisfying
\[
\lim_{k\rightarrow\infty}\left(\sum_{\boldsymbol{i}\in\mathcal{I}^k}\phi^t(T_{\boldsymbol{i}})^{1-q}p_{\boldsymbol{i}}^q\right)^{1/k}=1.
\]
This approach was introduced in \cite{Falconerlq1} where it was shown that for $q \in (1,2)$ the generalised $q$-dimensions of a self-affine measure is generically given by $d_q(T_1,\dots, T_N, \mu)$ in an appropriate sense.  See \cite{Falconerlq2} where further results along these lines were obtained for almost self-affine measures.  It is therefore of great interest to provide closed form expressions for $d_q(T_1,\dots, T_N, \mu)$ or at least to be able to estimate it effectively. We state the result using our notation and only in the planar case, although the higher dimensional case was also considered.

Let $T_1,\dots, T_N$ denote a collection of contracting non-singular $2\times 2$ upper triangular matrices and let $c_i, d_i$ denote the diagonal entries of the $i$th matrix. Define a function $P_0:[0,2] \times [0,1) \cup(1, \infty)\rightarrow [0, \infty)$ by
\[
P_0(t,q)=\begin{cases} 
\max\left\{ \sum_{i=1}^N p_i^q\ c_i^{t(1-q)},  \sum_{i=1}^N p_i^q\ d_i^{t(1-q)}\right\},  \qquad 0\leq t<1\\
\max\left\{ \sum_{i=1}^N p_i^q\left(c_i^{2-t}(c_i \ d_i)^{t-1}\right)^{1-q},  \sum_{i=1}^N p_i^q\left(d_i^{2-t}(c_i\ d_i)^{t-1}\right)^{1-q}\right\}, \qquad  1\leq t\leq 2\\
\end{cases}
\]
and, for each $q \in [0,1) \cup (1, \infty)$, let $u_0(q)$  be defined by $P_0(u_0(q),q)=1$, provided a solution exists and otherwise simply let $u_0(q)=2$. 

\begin{thm}$\cite[\textnormal{Theorem 4.1}]{Miao}$\label{Miao's Theorem}
Let $\mu$ be a planar self-affine measure generated by an IFS of upper triangular matrices as above.   Then for $q \in [0,1)$
\[
d_q(T_1,\dots, T_N,\mu)=u_0(q).
\]
\end{thm}
In the paper \cite{Miao}, this result was suggested to hold for all $q\geq 0$ ($q\neq 1$). The result appeared again in Miao's PhD thesis \cite[Theorem 3.11]{Miao1} in which he noted that, in fact, he could only establish the result for $q \in [0,1)$.  Miao conjectured that the result should still hold for $q>1$, see discussion leading up to \cite[Theorem 3.11]{Miao1}.    Our main result in this section, which is essentially an analogue of Theorem \ref{counterexample} adapted to this situation, proves that Theorem \ref{Miao's Theorem}, does \emph{not} hold for $q>1$ in general. 

We note that the approach in \cite{Miao,Miao1}  does provide a lower bound for $d_q(T_1,\dots, T_N,\mu)$ for $q>1$, that is, for all $q>1$,
\[
d_q(T_1,\dots, T_N,\mu)\geq u_0(q).
\]

\subsection{A family of counterexamples relating to generalised $q$-dimensions}

Before considering the range $q>1$ we note that a better lower bound than $u_0(q)$ is available simply by changing the maximum to a minimum in the definition of $P_0$, which is natural for $q>1$.  We define  $P_0^*:[0,2] \times [0,1) \cup (1, \infty) \rightarrow [0, \infty)$ by $P_0^*(t,q) = P_0(t,q)$ for $q \in [0,1)$ and for $q>1$ by
\[
P^*_0(t,q)=\begin{cases} 
\min\left\{ \sum_{i=1}^N p_i^q\ c_i^{t(1-q)},  \sum_{i=1}^N p_i^q\ d_i^{t(1-q)}\right\},  \qquad 0\leq t<1\\
\min\left\{ \sum_{i=1}^N p_i^q\left(c_i^{2-t}(c_i \ d_i)^{t-1}\right)^{1-q},  \sum_{i=1}^N p_i^q\left(d_i^{2-t}(c_i\ d_i)^{t-1}\right)^{1-q}\right\}, \qquad  1\leq t\leq 2.\\
\end{cases}
\]
Let $u(q)$  be defined by $P_0^*(u(q),q)=1$, provided a solution exists and otherwise simply let $u(q)=2$. Note that $u(q) = u_0(q)$ for $q \in [0,1)$ and $u(q) \geq u_0(q)$  for $q>1$ with strict inequality a possibility.  This inequality comes from the fact that the functions that we are taking the maximum or minimum of are \emph{increasing} in $t$ for $q>1$. We expect that when conjecturing a closed form expression for $d_q(T_1,\dots, T_N,\mu)$ for $q>1$, Miao \cite{Miao1} was thinking of $u(q)$ rather than $u_0(q)$.   \\

\begin{lem}
For all $q\geq 0$ ($q \neq 1$) we  have
\[
d_q(T_1,\dots, T_N,\mu)\geq u(q).
\]
\end{lem}
\begin{proof}
It suffices to only consider the range $q>1$ since for $q<1$ this result is covered by \cite{Miao, Miao1}.   Write $\alpha_1(\boldsymbol{i})\geq\alpha_2(\boldsymbol{i})$ for the singular values of the matrix $T_{\boldsymbol{i}}$. Firstly suppose that $0\leq u(q)<1$ and therefore
\[
\sum_{\boldsymbol{i}\in\mathcal{I}^k}\phi^{u(q)}(T_{\boldsymbol{i}})^{1-q}p_{\boldsymbol{i}}^q=\sum_{\boldsymbol{i}\in\mathcal{I}^k}\alpha_1(\boldsymbol{i})^{u(q)(1-q)}p_{\boldsymbol{i}}^q.
\]
By definition of $\alpha_1(\boldsymbol{i})$ we have $\alpha_1(\boldsymbol{i})= \max\{c_{\boldsymbol{i}}, d_{\boldsymbol{i}}\}$  and since $u(q)(1-q)<0$ it follows that $\alpha_1(\boldsymbol{i})^{u(q)(1-q)}\leq \min\{c_{\boldsymbol{i}}^{u(q)(1-q)}, d_{\boldsymbol{i}}^{u(q)(1-q)}\}$. Therefore
\[
\sum_{\boldsymbol{i}\in\mathcal{I}^k}\phi^{u(q)}(T_{\boldsymbol{i}})^{1-q}p_{\boldsymbol{i}}^q\leq\min\left\{\sum_{\boldsymbol{i}\in\mathcal{I}^k}c_{\boldsymbol{i}}^{u(q)(1-q)}p_{\boldsymbol{i}}^q, \sum_{\boldsymbol{i}\in\mathcal{I}^k}d_{\boldsymbol{i}}^{u(q)(1-q)}p_{\boldsymbol{i}}^q\right\} = P_0^*(u(q),q) = 1.
\]
where we have used the fact that $c_{\boldsymbol{i}}$ and $d_{\boldsymbol{i}}$ are multiplicative in $\boldsymbol{i}$. Therefore, for $t = u(q)$,
\[
\lim_{k\rightarrow\infty}\left(\sum_{\boldsymbol{i}\in\mathcal{I}^k}\phi^t(T_{\boldsymbol{i}})^{1-q}p_{\boldsymbol{i}}^q\right)^{1/k} \leq 1
\]
and since the expression on the left is \emph{increasing} in $t$ (since $q>1$)
\[
d_q(T_1,\dots, T_N,\mu)\geq u(q).
\]
If $1\leq u(q)<2$,  then the proof follows similarly noting
\[
\phi^{u(q)}(T_{\boldsymbol{i}})^{1-q} = \left(\alpha_1(\boldsymbol{i}) \alpha_2(\boldsymbol{i})^{u(q)-1}\right)^{1-q}\leq \min\left\{  \left(c_{\boldsymbol{i}}^{2-t}(c_{\boldsymbol{i}} \ d_{\boldsymbol{i}})^{t-1}\right)^{1-q},  \ \left(d_{\boldsymbol{i}}^{2-t}(c_{\boldsymbol{i}}\ d_{\boldsymbol{i}})^{t-1}\right)^{1-q}\right\}.
\]
 We leave the details to the reader.
\end{proof}

Despite this simple improvement on the lower bound, we prove that $d_q(T_1,\dots, T_N,\mu)$ is still \emph{not} generally equal to $u(q)$  for $q>1$.\\

\begin{thm}\label{Miao Counterexample}
Let $c, d$ be such that  $c>d>0$ and $c+d\leq 1$. Let $\mu$ be the self-affine measure defined by the probability vector $(1/2, 1/2)$ and the diagonal system consisting of the two maps, $T_{1}$ and $T_{2}$, defined by
\[
T_{1}(x, y)=\begin{pmatrix}
  c & 0  \\
  0 &  d
 
 \end{pmatrix}\begin{pmatrix}
  x  \\
  y 
 
 \end{pmatrix}
 \qquad \text{ and  }  \qquad 
T_{2}(x, y)=\begin{pmatrix}
  d & 0  \\
  0 &  c
 
 \end{pmatrix}\begin{pmatrix}
  x  \\
  y 
 
 \end{pmatrix}+\begin{pmatrix}
  1-d  \\
  1-c 
 
 \end{pmatrix}.
\]
For $q>1$ let $u(q)$ be defined by $P^*_0(u(q),q)=1$ as in the statement of Theorem \ref{Miao's Theorem}, that is, $u(q)$ is the unique solution of 
\[
c^{u(q)(1-q)}2^{-q}+d^{u(q)(1-q)}2^{-q}=1.
\]
Then, for all $q>1$,
\[
d_q(T_1, T_2, \mu)>u(q).
\] 
More precisely, for all $q>1$,
\[
d_q(T_1, T_2, \mu)\geq u(q)+\frac{2\log \left(\frac{2(c/d)^{u(q)(q-1)/2}}{(c/d)^{u(q)(q-1)}+1}\right)}{(q-1)\log(cd)}.
\]
\end{thm}
\begin{proof}
We adapt the proof of Theorem \ref{counterexample}. Let $q>1$, $k$ be odd, and consider the following sum
\[
\sum_{\boldsymbol{i}\in\mathcal{I}^k}\phi^{u(q)}(T_{\boldsymbol{i}})^{1-q}p_{\boldsymbol{i}}^q= \sum_{\boldsymbol{i}\in\mathcal{I}^k}\alpha_1(\boldsymbol{i})^{u(q)(1-q)}2^{-kq},
\]
noting that $u(q) \leq u(0) \leq 1$.  As before we see that for $\boldsymbol{i}\in\mathcal{I}^k$  if $T_1$ appears $i$ times in the composition of  $T_{\boldsymbol{i}}$ and $T_2$ appears $k-i$ times, then, since $c>d$,
\[
\alpha_1(\boldsymbol{i})=c^{\max\{i, k-i\}}\times d^{\min\{i, k-i\}}
\]
and so the above is equal to
\[
\sum_{i=0}^{\lfloor k/2 \rfloor}\binom{k}{i}\left(c^{k-i}d^{i}\right)^{u(q)(1-q)}2^{-kq}+ \sum_{i=\lceil k/2 \rceil}^{k}\binom{k}{i}\left(d^{k-i}c^{i}\right)^{u(q)(1-q)}2^{-kq}.
\]
We again define $X_k^{q}$ and $Y_k^{q}$ to be the left and right parts of the above.  
Continuing with  exactly the same approach as in the proof of Theorem \ref{counterexample} and applying Theorem \ref{binomial}, where in this case $x= (c/d)^{u(q)(q-1)}>1$, we find that
\[
\lim_{k\rightarrow\infty}\left(\sum_{\boldsymbol{i}\in\mathcal{I}^k}\phi^{u(q)}(T_{\boldsymbol{i}})^{1-q}p_{\boldsymbol{i}}^q\right)^{1/k}  = \lim_{k\rightarrow\infty}\left(X_k^{q}+Y_k^{q}\right)^{1/k}=\frac{2(c/d)^{u(q)(q-1)/2}}{(c/d)^{u(q)(q-1)}+1} =:\delta<1.
\]
Recall that since $1-q<0$, it follows in this setting that 
\[
\lim_{k\rightarrow\infty}\left(\sum_{\boldsymbol{i}\in\mathcal{I}^k}\phi^t(T_{\boldsymbol{i}})^{1-q}p_{\boldsymbol{i}}^q\right)^{1/k} 
\]
is a strictly \emph{increasing} function of $t$ and therefore
\[
d_q(T_1, T_2, \mu)>u(q)
\]
as required.  We can upgrade this result to get the stated quantitative lower bound by considering the definition of $d_q(T_1, T_2, \mu)$ more closely.   For  $k \geq 1$ and $\boldsymbol{i}\in\mathcal{I}^k$, we have $\alpha_1(\boldsymbol{i}) \geq (cd)^{ k/2  }$ and therefore, for $\varepsilon = d_q(T_1, T_2, \mu)-u(q)>0$,
\begin{eqnarray*}
\delta = \lim_{k\rightarrow\infty}\left(\sum_{\boldsymbol{i}\in\mathcal{I}^k}\phi^{u(q)}(T_{\boldsymbol{i}})^{1-q}p_{\boldsymbol{i}}^q\right)^{1/k} &=& \lim_{k\rightarrow\infty}\left( \sum_{\boldsymbol{i}\in\mathcal{I}^k}\alpha_1(\boldsymbol{i})^{u(q)(1-q)}2^{-kq} \right)^{1/k} \\
  &\geq & (cd)^{\varepsilon(q-1)/2} \lim_{k\rightarrow\infty}\left( \sum_{\boldsymbol{i}\in\mathcal{I}^k}\alpha_1(\boldsymbol{i})^{d_q(T_1, T_2, \mu)(1-q)}2^{-kq} \right)^{1/k}  \\
 &=& (cd)^{\varepsilon(q-1)/2}
\end{eqnarray*}
and therefore
\[
d_q(T_1, T_2, \mu)-u(q) = \varepsilon  \geq \frac{2\log \delta}{(q-1)\log(cd)}
\]
which proves the theorem.
\end{proof}

\subsection{New closed form bounds for generalised dimensions}

Despite the fact that $d_q(T_1,\dots, T_N, \mu)$ is not given by the value predicted by Falconer-Miao \cite{Miao, Miao1} $q>1$, we can still find upper bounds in the case when our matrices are diagonal by following the approach of Section \ref{Lower Bounds}.   To simplify notation and aid readability, we only pursue such bounds in the planar case but higher dimensional analogues could be proved similarly.  For convenience here we let $\mathcal{I}$ denote the set $\{1,\dots,N\}$. We also let  $t_1, t_2, s_1, s_2$ be defined by the following equations:
\begin{align*}
&\sum_{i=1}^N p_i^q\ c_i^{t_1(1-q)}=1, \qquad  \sum_{i=1}^N p_i^q\ d_i^{t_2(1-q)}=1,\\
&\sum_{i=1}^N p_i^q\left(c_i^{2-s_1}(c_i \ d_i)^{s_1-1}\right)^{1-q}=1, \qquad \sum_{i=1}^N p_i^q\left(d_i^{2-s_2}(c_i\ d_i)^{s_2-1}\right)^{1-q}=1,
 \end{align*}
and, as in the previous section, define $u(q)$ by $P_0^*(u(q), q)=1$.  We may assume that $u(q)<2$, as otherwise there is nothing to prove,  and we note that $u(q)$ is always equal to one of $t_1, t_2, s_1, s_2$. Once again we write $x^+$ for the maximum of $x\in\mathbb{R}$ and $0$.\\

\begin{thm}\label{Upper Bound Theorem}
Let $\mu$ be a self-affine measure generated by a diagonal system in $\mathbb{R}^2$ and assume that $q>1$.  

$(a)$ If $1\leq u(q)< 2$ then 
\[
d_q(T_1,\dots, T_N,\mu)\leq \min\{U_1(q), U_2(q)\}
\]
where
\[
U_1(q)= s_1+\left((2-s_1)\frac{\sum_{i\in\mathcal{I}}p_{i}^q \ c_{i}^{1-q}\ d_{i}^{(s_1-1)(1-q)}\log(c_i/d_i)}{\sum_{i\in\mathcal{I}}p_{i}^q \ c_{i}^{1-q}\ d_{i}^{(s_1-1)(1-q)}\log(c_i)}\right)^+
\]
and
\[
U_2(q)= s_2+\left((2-s_2)\frac{\sum_{i\in\mathcal{I}}p_{i}^q \ d_{i}^{1-q}\ c_{i}^{(s_2-1)(1-q)}\log(d_i/c_i)}{\sum_{i\in\mathcal{I}}p_{i}^q \ d_{i}^{1-q}\ c_{i}^{(s_2-1)(1-q)}\log(d_i)}\right)^+.
\]
Here 
\[
\frac{\sum_{i\in\mathcal{I}}p_{i}^q \ c_{i}^{1-q}\ d_{i}^{(s_1-1)(1-q)}\log(c_i/d_i)}{\sum_{i\in\mathcal{I}}p_{i}^q \ c_{i}^{1-q}\ d_{i}^{(s_1-1)(1-q)}\log(c_i)}
\]
and
\[
\frac{\sum_{i\in\mathcal{I}}p_{i}^q \ d_{i}^{1-q}\ c_{i}^{(s_2-1)(1-q)}\log(d_i/c_i)}{\sum_{i\in\mathcal{I}}p_{i}^q \ d_{i}^{1-q}\ c_{i}^{(s_2-1)(1-q)}\log(d_i)}
\]
are strictly less than 1, which we emphasise as it ensures that this is a strictly better bound than $d_q(T_1,\dots, T_N,\mu)\leq 2$.

$(b)\ (i)$ If $0\leq u(q)< 1$ then
\[
d_q(T_1,\dots, T_N,\mu)\leq\min\{V_1(q), V_2(q)\}
\]
where
\[
V_1(q)= t_1+\left(t_1\frac{\sum_{i\in\mathcal{I}}p_i^q\ c_i^{t_1(1-q)}\log(c_i/d_i) }{\sum_{i\in\mathcal{I}}p_i^q\ c_i^{t_1(1-q)} \log(d_i)}\right)^+
\]
and
\[
V_2(q)= t_2+\left(t_2\frac{\sum_{i\in\mathcal{I}}p_i^q\ d_i^{t_2(1-q)}\log(d_i/c_i) }{\sum_{i\in\mathcal{I}}p_i^q\ d_i^{t_2(1-q)} \log(c_i)}\right)^+
\]
provided $\min\{V_1(q), V_2(q)\} \leq 1$. 

$(ii)$ If $\min\{V_1(q), V_2(q)\} > 1$, then 
\[
d_q(T_1,\dots, T_N,\mu)\leq\min\{W_1(q), W_2(q)\}
\]
where 
\[
W_1(q)=t_1+\max\{A(q), C(q)\}^+
\]
and
\[
W_2(q)=t_2+\max\{B(q), D(q)\}^+
\]
and where
\begin{align*}
&A(q)=(1-t_1)\frac{\sum_{i\in\mathcal{I}}p_i^q\ c_i^{t_1(1-q)}\log(d_i/c_i)}{\sum_{i\in\mathcal{I}}p_i^q\ c_i^{t_1(1-q)} \log(d_i)}\\
&B(q)=(1-t_2)\frac{\sum_{i\in\mathcal{I}}p_i^q\ d_i^{t_2(1-q)}\log(c_i/d_i)}{\sum_{i\in\mathcal{I}}p_i^q\ d_i^{t_2(1-q)} \log(c_i)} \\
&C(q)=\frac{\sum_{i\in\mathcal{I}}p_i^q\ c_i^{t_1(1-q)}\log(c_i/d_i)}{\sum_{i\in\mathcal{I}}p_i^q\ c_i^{t_1(1-q)} \log(c_i)}  \\
&D(q)=\frac{\sum_{i\in\mathcal{I}}p_i^q\ d_i^{t_2(1-q)}\log(d_i/c_i)}{\sum_{i\in\mathcal{I}}p_i^q\ d_i^{t_2(1-q)} \log(d_i)}.
\end{align*}



\end{thm}

\begin{proof}
The proof follows the strategy of the proof of Theorem \ref{Lower Bound Theorem} and so we suppress some common details.  Let $\lbrace\theta_{i}\rbrace_{i\in\mathcal{I}}$ denote an arbitrary probability vector and, for each $k\in\mathbb{N}$, define  $n(k)\in\mathbb{N}$ by
\[
n(k)=\sum_{i\in\mathcal{I}}\lfloor{\theta_{i}k}\rfloor.
\end{equation*}
Recall that $k-|\mathcal{I}|\leq n(k)\leq k$. We again consider the $n(k)$th iteration of $\mathcal{I}$ and define
\[
\mathcal{J}_{k}=\left\lbrace\boldsymbol{j}=(j_{1},\dots,j_{n(k)})\in\mathcal{I}^{n(k)}:\#\lbrace m:j_{m}=i\rbrace=\lfloor\theta_{i}k\rfloor\ \textnormal{for each}\ i\in\mathcal{I}\right\rbrace
\]
noting, again, that
\[
|\mathcal{J}_{k}|=\frac{n(k)!}{\prod_{i\in\mathcal{I}}\lfloor\theta_{i}k\rfloor !}.
\]
We also define numbers $c$, $d$ and $p$ by 
\[
c=\prod_{i\in\mathcal{I}}c_{i}^{\lfloor\theta_{i}k\rfloor}, \qquad d=\prod_{i\in\mathcal{I}}d_{i}^{\lfloor\theta_{i}k\rfloor}, \qquad p=\prod_{i\in\mathcal{I}}p_{i}^{\lfloor\theta_{i}k\rfloor}.
\]
$(a)$ Firstly we shall consider the case when $1\leq u(q)<2$, so in this case $u(q)$ is given by either $s_1$ and $s_2$, which are defined above. Also assume that $\prod_{i\in\mathcal{I}}c_{i}^{\theta_i}>\prod_{i\in\mathcal{I}}d_{i}^{\theta_i}$. We know from the proof of Theorem \ref{Lower Bound Theorem} that this condition implies that $c>d$ for $k$ sufficiently large.  We then have that for all $\boldsymbol{i}\in\mathcal{J}_{k}$ and $s >0$ that
\[
\phi^s(T_{\boldsymbol{i}})^{1-q}p_{\boldsymbol{i}}^q= (c\ d^{s-1})^{1-q}p^q=p^q\ c^{1-q}\ d^{(s-1)(1-q)}
\]
which by definition of $p,c$ and $d$ we may write as
\[
\phi^s(T_{\boldsymbol{i}})^{1-q}p_{\boldsymbol{i}}^q=\prod_{i\in\mathcal{I}}\left(p_{i}^q \ c_{i}^{1-q}\ d_{i}^{(s-1)(1-q)}\right)^{\lfloor\theta_{i}k\rfloor}.
\]
Using exactly the same reasoning as in the proof of Theorem \ref{Lower Bound Theorem} (simply replacing $p_{i}^q \ c_{i}^{\tau_{1}(q)}\ d_{i}^{s-\tau_{1}(q)}$ by $p_{i}^q \ c_{i}^{1-q}\ d_{i}^{(s-1)(1-q)}$) we may show that
\begin{align*}
\log\left(\left(\sum_{\boldsymbol{i}\in\mathcal{I}^{n(k)}}\phi^s(T_{\boldsymbol{i}})^{1-q}p_{\boldsymbol{i}}^q\right)^{1/{n(k)}}\right)&\geq  \log\left(\frac{k-|\mathcal{I}|}{k}\right)-\frac{1}{k-|\mathcal{I}|}\sum_{i\in\mathcal{I}}\log\theta_{i}k\\ 
&\qquad\qquad\qquad\qquad +\sum_{i\in\mathcal{I}}\theta_{i} \log\left( \frac{p_{i}^q \ c_{i}^{1-q}\ d_{i}^{(s-1)(1-q)}}{\theta_{i}}\right)
\end{align*}
which converges to
\[
\sum_{i\in\mathcal{I}}\theta_{i} \log\left( \frac{p_{i}^q \ c_{i}^{1-q}\ d_{i}^{(s-1)(1-q)}}{\theta_{i}}\right)
\]
as $k\rightarrow\infty$.  If this is greater than or equal to $0$ then we get that
\[
\lim_{k\rightarrow\infty}\left(\sum_{\boldsymbol{i}\in\mathcal{I}^k}\phi^s(T_{\boldsymbol{i}})^{1-q}p_{\boldsymbol{i}}^q\right)^{1/k}\geq 1
\]
and therefore
\[
d_q(T_1,\dots, T_N,\mu)\leq s.
\]
This follows  because when $q>1$ the above limit is a strictly increasing function of $s$ (as opposed to when $0<q<1$, when it is a strictly decreasing function of $s$). As  before we can use a very similar argument when $\prod_{i\in\mathcal{I}}c_{i}^{\theta_i}\leq\prod_{i\in\mathcal{I}}d_{i}^{\theta_i}$.  Combining these cases we find that
\[
\begin{split}
d_q(T_1,\dots, T_N,\mu)\leq \inf \Bigg\{ s:  \textnormal{there exists a probability vector}\ \lbrace\theta_{i}\rbrace_{i\in\mathcal{I}} \ \textnormal{such that either}\\ (1) \ \prod_{i\in\mathcal{I}}c_{i}^{\theta_i}>\prod_{i\in\mathcal{I}}d_{i}^{\theta_i} \ \textnormal{and}\sum_{i\in\mathcal{I}}\theta_{i} \log\left( \frac{p_{i}^q \ c_{i}^{1-q}\ d_{i}^{(s-1)(1-q)}}{\theta_{i}}\right)\geq 0\\ \textnormal{or}\ (2) \ \prod_{i\in\mathcal{I}}c_{i}^{\theta_i}<\prod_{i\in\mathcal{I}}d_{i}^{\theta_i} \ \textnormal{and}\sum_{i\in\mathcal{I}}\theta_{i} \log\left( \frac{p_{i}^q \ d_{i}^{1-q}\ c_{i}^{(s-1)(1-q)}}{\theta_{i}}\right)\geq 0 \\
\textnormal {or}\ (3) \ \prod_{i\in\mathcal{I}}c_{i}^{\theta_i}=\prod_{i\in\mathcal{I}}d_{i}^{\theta_i}\ \textnormal{and both}\sum_{i\in\mathcal{I}}\theta_{i} \log\left( \frac{p_{i}^q \ c_{i}^{1-q}\ d_{i}^{(s-1)(1-q)}}{\theta_{i}}\right)\geq 0\\ \textnormal{and}\sum_{i\in\mathcal{I}}\theta_{i} \log\left( \frac{p_{i}^q \ d_{i}^{1-q}\ c_{i}^{(s-1)(1-q)}}{\theta_{i}}\right)\geq 0&\Bigg\}.
\end{split}
\]
Once again, we have the freedom to choose a probability vector. Natural choices here would be to take either $\{p_i^q\left(c_i^{2-s_1}(c_i\ d_i)^{s_1-1}\right)^{1-q}\}_{i\in\mathcal{I}}$   or $\{p_i^q\left(d_i^{2-s_2}(c_i\ d_i)^{s_2-1}\right)^{1-q}\}_{i\in\mathcal{I}}$, which by definition of $s_1$ and $s_2$ are indeed probability vectors. Recall that $u(q)$ is given by either $s_1$ or $s_2$. Choose
\[
\{\theta_{i}\}_{i\in\mathcal{I}}=\left\{ p_i^q\left(c_i^{2-s_1}(c_i\ d_i)^{s_1-1}\right)^{1-q}\right\}_{i\in\mathcal{I}}=\left\{p_{i}^q \ c_{i}^{1-q}\ d_{i}^{(s_1-1)(1-q)}\right\}_{i\in\mathcal{I}}.
\]
We also replace $s$ in the above by $s_1+\varepsilon$, where $\varepsilon\geq 0$ is small enough so that $1<s_1+\varepsilon< 2$ (note this clearly does not affect any of the above calculations).  We want to investigate how small we can choose $\varepsilon$.  We again require two conditions to hold, the first of which holds trivially since
\begin{align*}
&\sum_{i\in\mathcal{I}}p_{i}^q \ c_{i}^{1-q}\ d_{i}^{(s_1-1)(1-q)} \log\left( \frac{p_{i}^q \ c_{i}^{1-q}\ d_{i}^{(s_1+\varepsilon-1)(1-q)}}{p_{i}^q \ c_{i}^{1-q}\ d_{i}^{(s_1-1)(1-q)}}\right)\\ =&\sum_{i\in\mathcal{I}}p_{i}^q \ c_{i}^{1-q}\ d_{i}^{(s_1-1)(1-q)} \log\left( d_i^{\varepsilon(1-q)}\right)\geq 0.
\end{align*}
For the second condition to hold, we require
\[
\sum_{i\in\mathcal{I}}p_{i}^q \ c_{i}^{1-q}\ d_{i}^{(s_1-1)(1-q)} \log\left( \frac{p_{i}^q \ d_{i}^{1-q}\ c_{i}^{(s_1+\varepsilon-1)(1-q)}}{p_{i}^q \ c_{i}^{1-q}\ d_{i}^{(s_1-1)(1-q)}}\right)\geq 0
\]
which, rearranging, is equivalent to
\[
\varepsilon\geq (2-s_1)\frac{\sum_{i\in\mathcal{I}}p_{i}^q \ c_{i}^{1-q}\ d_{i}^{(s_1-1)(1-q)}\log(c_i/d_i)}{\sum_{i\in\mathcal{I}}p_{i}^q \ c_{i}^{1-q}\ d_{i}^{(s_1-1)(1-q)}\log(c_i)}.
\]
This implies that 
\[
d_q(T_1,\dots, T_N,\mu)\leq s_1+\left((2-s_1)\frac{\sum_{i\in\mathcal{I}}p_{i}^q \ c_{i}^{1-q}\ d_{i}^{(s_1-1)(1-q)}\log(c_i/d_i)}{\sum_{i\in\mathcal{I}}p_{i}^q \ c_{i}^{1-q}\ d_{i}^{(s_1-1)(1-q)}\log(c_i)}\right)^+ = U_1(q).
\]
Note if the right hand side of the above lower bound for $\varepsilon$ is negative then we take $\varepsilon=0$, which is why the ${}^+$ appears.    Finally note that 
\[
\frac{\sum_{i\in\mathcal{I}}p_{i}^q \ c_{i}^{1-q}\ d_{i}^{(s_1-1)(1-q)}\log(c_i/d_i)}{\sum_{i\in\mathcal{I}}p_{i}^q \ c_{i}^{1-q}\ d_{i}^{(s_1-1)(1-q)}\log(c_i)}=1-\frac{\sum_{i\in\mathcal{I}}p_{i}^q \ c_{i}^{1-q}\ d_{i}^{(s_1-1)(1-q)}\log(d_i)}{\sum_{i\in\mathcal{I}}p_{i}^q \ c_{i}^{1-q}\ d_{i}^{(s_1-1)(1-q)}\log(c_i)}< 1
\]
so our upper bound is an improvement on
\[
d_q(T_1,\dots, T_N,\mu)\leq 2.
\]
The other upper bound $d_q(T_1,\dots, T_N,\mu)\leq  U_2(q)$ is proved similarly and relies on the other natural choice of $\{\theta_i\}$.

$(b)$ We shall now assume that $0\leq u(q)<1$, so here $u(q)$ is given by either $t_1$ or $t_2$, defined above. Considering again the $n(k)$th iteration of $\mathcal{I}$ and first supposing that $\prod_{i\in\mathcal{I}}c_{i}^{\theta_i}>\prod_{i\in\mathcal{I}}d_{i}^{\theta_i}$, then for all $\boldsymbol{i}\in\mathcal{J}_{k}$  and $s \in [0,1]$ 
\[
\phi^s(T_{\boldsymbol{i}})^{1-q}p_{\boldsymbol{i}}^q=p^q\ c^{s(1-q)}.
\]
We  use exactly the same reasoning as above and find that in this case
\[
\begin{split}
d_q(T_1,\dots, T_N,\mu)\leq \inf \Bigg\{ s \in [0,1]  :  \textnormal{there exists a probability vector}\ \lbrace\theta_{i}\rbrace_{i\in\mathcal{I}} \ \textnormal{such that either}\\ (1) \ \prod_{i\in\mathcal{I}}c_{i}^{\theta_i}>\prod_{i\in\mathcal{I}}d_{i}^{\theta_i} \ \textnormal{and}\sum_{i\in\mathcal{I}}\theta_{i} \log\left( \frac{p_{i}^q \ c_{i}^{s(1-q)}}{\theta_{i}}\right)\geq 0\\ \textnormal{or}\ (2) \ \prod_{i\in\mathcal{I}}c_{i}^{\theta_i}<\prod_{i\in\mathcal{I}}d_{i}^{\theta_i} \ \textnormal{and}\sum_{i\in\mathcal{I}}\theta_{i} \log\left( \frac{p_{i}^q \ d_{i}^{s(1-q)}}{\theta_{i}}\right)\geq 0 \\
\textnormal {or}\ (3) \ \prod_{i\in\mathcal{I}}c_{i}^{\theta_i}=\prod_{i\in\mathcal{I}}d_{i}^{\theta_i}\ \textnormal{and both}\sum_{i\in\mathcal{I}}\theta_{i} \log\left( \frac{p_{i}^q \ c_{i}^{s(1-q)}}{\theta_{i}}\right)\geq 0\\ \textnormal{and}\sum_{i\in\mathcal{I}}\theta_{i} \log\left( \frac{p_{i}^q \ d_{i}^{s(1-q)}}{\theta_{i}}\right)\geq 0&\Bigg\}.
\end{split}
\]
Note the complication here that we require $s \leq 1$ because we assume the singular value function takes the form $\alpha_1^s$.  This is what leads to the awkward extra case in the $u(q)<1$ setting.

Again, there are two natural choices for probability vector $\{\theta_i\}$, the first of which is
\[
\{\theta_{i}\}_{i\in\mathcal{I}}=\left\{p_i^q\ c_i^{t_1(1-q)}\right\}_{i\in\mathcal{I}}.
\]
We replace $s$ by $t_1+\varepsilon$ in the above, where $\varepsilon\geq 0$. Once again we would like to see how small it is possible to take $\varepsilon$. We must to consider two cases: when $\varepsilon$ can be taken sufficiently small so that $t_1+\varepsilon<1$ and when $1\leq t_1+\varepsilon<2$ (this will affect which form of the singular value function we can use).

$(i)$ Firstly suppose we can take $\varepsilon$ sufficiently small so that $t_1+\varepsilon<1$. We require two conditions to hold, the first of which is trivial since
\[
\sum_{i\in\mathcal{I}}p_{i}^q \ c_{i}^{t_1(1-q)} \log\left( \frac{p_{i}^q \ c_{i}^{(s+\varepsilon)(1-q)}}{p_{i}^q \ c_{i}^{t_1(1-q)}}\right) = \sum_{i\in\mathcal{I}}p_{i}^q \ c_{i}^{t_1(1-q)} \log(c_i^{\varepsilon(1-q)})\geq 0.
\]
For the second condition to hold, we require
\[
\sum_{i\in\mathcal{I}}p_i^q\ c_i^{t_1(1-q)} \log\left( \frac{p_{i}^q \ d_{i}^{(t_1+\varepsilon)(1-q)}}{p_i^q\ c_i^{t_1(1-q)}}\right)\geq 0
\]
which is equivalent to
\[
\varepsilon\geq t_1\frac{\sum_{i\in\mathcal{I}}p_i^q\ c_i^{t_1(1-q)}\log(c_i/d_i)}{\sum_{i\in\mathcal{I}}p_i^q\ c_i^{t_1(1-q)} \log(d_i)}.
\]
This implies that
\[
d_q(T_1,\dots, T_N,\mu)\leq t_1+\left(t_1\frac{\sum_{i\in\mathcal{I}}p_i^q\ c_i^{t_1(1-q)}\log(c_i/d_i) }{\sum_{i\in\mathcal{I}}p_i^q\ c_i^{t_1(1-q)} \log(d_i)}\right)^+.
\]
$(ii)$ Now suppose that we cannot take $\varepsilon$ sufficiently small so that $t_1+\varepsilon<1$, so that we instead have to consider what happens when $1\leq t_1+\varepsilon< 2$. In this case we will still be using the same choice of probability vector but we will be using the form of the singular value function in the range $[1,2]$, that is $\alpha_1\alpha_2^{s-1}$, and we refer to the general upper bound in the case $1 \leq u(q)<2$ given above.

As usual we require two conditions to hold simultaneously, but this time neither condition is trivial.  We require
\[
\sum_{i\in\mathcal{I}}p_i^q\ c_i^{t_1(1-q)} \log\left( \frac{p_{i}^q \ c_{i}^{1-q}\ d_{i}^{(t_1+\varepsilon-1)(1-q)}}{p_i^q\ c_i^{t_1(1-q)}}\right)\geq 0
\]
which is equivalent to $\varepsilon\geq A(q)$, where
\[
A(q)= (1-t_1)\frac{\sum_{i\in\mathcal{I}}p_i^q\ c_i^{t_1(1-q)}\log(d_i/c_i)}{\sum_{i\in\mathcal{I}}p_i^q\ c_i^{t_1(1-q)} \log(d_i)}.
\]
We also require
\[
\sum_{i\in\mathcal{I}}p_i^q\ c_i^{t_1(1-q)} \log\left( \frac{p_{i}^q \ d_{i}^{1-q}\ c_{i}^{(t_1+\varepsilon-1)(1-q)}}{p_i^q\ c_i^{t_1(1-q)}}\right)\geq 0
\]
which is equivalent to $\varepsilon\geq C(q)$, where
\[
C(q)= \frac{\sum_{i\in\mathcal{I}}p_i^q\ c_i^{t_1(1-q)}\log(c_i/d_i)}{\sum_{i\in\mathcal{I}}p_i^q\ c_i^{t_1(1-q)} \log(c_i)}.
\]
Thus we may conclude in this instance that
\[
d_q(T_1,\dots, T_N,\mu)\leq t_1+\max\{A(q), C(q)\}^+ = W_1(q).
\]
The other upper bound, $W_2(q)$, can be derived similarly.
\end{proof}

As a corollary to the above, we present simple conditions that ensure $d_q(T_1,\dots, T_N,\mu)=u(q)$, that is, for the Theorem of Falconer-Miao to hold when $q>1$.\\








\begin{cor}\label{Equality}
Consider the diagonal system of Theorem \ref{Upper Bound Theorem} and $q>1$.  First suppose that $1<u(q)\leq 2$. If $u(q) = s_1$ and 
\[
\sum_{i\in\mathcal{I}}p_{i}^q \ c_{i}^{1-q}\ d_{i}^{(s_1-1)(1-q)}\log(c_i/d_i)\geq 0,
\]
then $d_q(T_1,\dots, T_N,\mu)=u(q) = s_1$.  If  $u(q) = s_2$  and
\[
\sum_{i\in\mathcal{I}}p_{i}^q \ d_{i}^{1-q}\ c_{i}^{(s_2-1)(1-q)}\log(d_i/c_i)\geq 0,
\]
then  $d_q(T_1,\dots, T_N,\mu)=u(q)=s_2$. Secondly, suppose that $0<u(q)\leq 1$.  If  $u(q) = t_1$  and
\[
\sum_{i\in\mathcal{I}}p_i^q\ c_i^{t_1(1-q)}\log(c_i/d_i)\geq 0,
\]
then  $d_q(T_1,\dots, T_N,\mu)=u(q)=t_1$.  If    $u(q) = t_2$  and
\[
\sum_{i\in\mathcal{I}}p_i^q\ d_i^{t_2(1-q)}\log(d_i/c_i)\geq 0,
\]
then $d_q(T_1,\dots, T_N,\mu)=u(q)=t_2$. 

In particular, if  $c_i \geq d_i$ for all $i \in \mathcal{I}$ or $c_i \leq d_i$ for all $i \in \mathcal{I}$, then $d_q(T_1,\dots, T_N,\mu)=u(q)$. 
\end{cor}

\begin{proof}
This follows from Theorem \ref{Upper Bound Theorem}, noting in each instance that if one of these conditions holds then we may choose $\varepsilon=0$.
\end{proof}

\subsection{An example}

Here we present an example of a diagonal system to which Corollary \ref{Equality} can be applied. We take $p_1=4/5, p_2=1/10, p_3=1/10$ as our probability vector and define three maps by choosing $c_1=2/5, c_2=3/10, c_3=3/10$ and  $d_1=3/10, d_2=2/5, d_3=3/10$. For $q \in [0,5]$, we have $0<u(q)\leq 1$ and
\[
\sum_{i\in\mathcal{I}}p_i^q\ c_i^{t_1(1-q)}\log(c_i/d_i)\geq 0
\]
which means the first condition from Corollary \ref{Equality} is satisfied.  Therefore $d_q(T_1, T_2, T_3,\mu) = u(q) = t_1$ for $q \in [0,5]$ by  Corollary \ref{Equality}, see Figure \ref{miaocasefig3}.
\begin{figure}[H]
  \centering
\includegraphics[width=\textwidth]{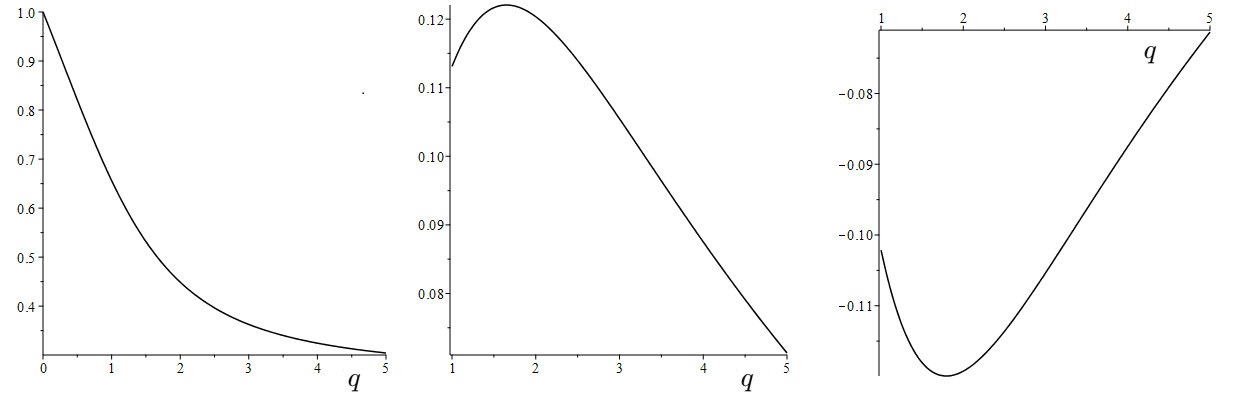}
\caption{Left: plot of $d_q(T_1, T_2, T_3,\mu)=u(q)$. Previously this formula was only known for $0<q<1$, see \cite{Miao, Miao1}. Middle:  plot of the first condition from Corollary \ref{Equality}, which is satisfied for the whole range of $q$.  Right: plot of the second condition from Corollary \ref{Equality}, which is not satisfied.} 
\label{miaocasefig3}
\end{figure}
Observe that the value at $q=0$ gives the affinity dimension of the set our measure is supported on, which in this case is 1. Also recall that, by Falconer's result  \cite[Theorem 6.2]{Falconerlq1}, the generalised $q$-dimensions of $\mu$ are given by $d_q(T_1, T_2, T_3,\mu)$  for $1<q\leq 2$ almost surely upon if  randomising the translation vectors, provided the norms of the matrices are strictly less than $1/2$.
\begin{figure}[H]
  \centering
\includegraphics[width=\textwidth]{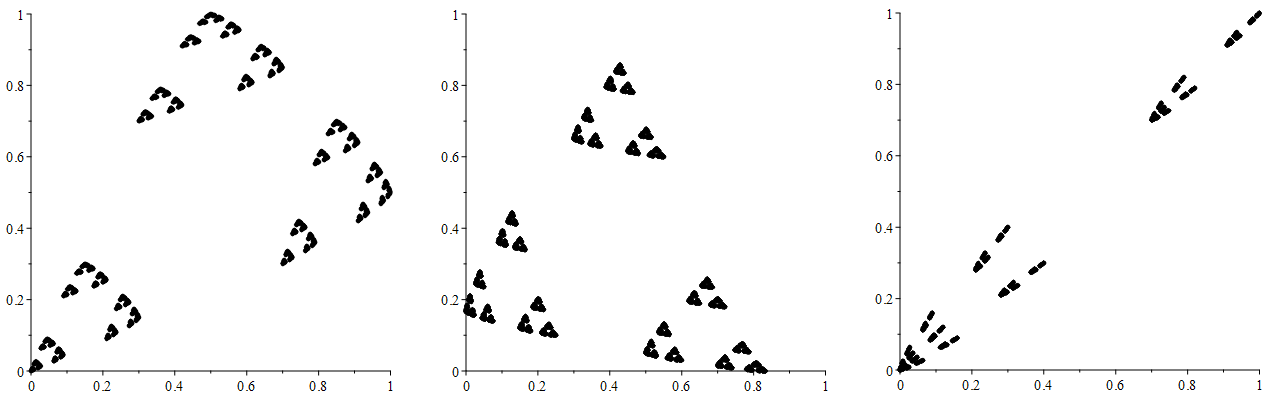}
\caption{Three self-affine measures generated by the set of matrices and probabilities given above.  The translations have been chosen randomly in each case and almost surely the generalised $q$-dimensions in each case are given by $u(q) = d_q(T_1, T_2, T_3,\mu)$ above  for $1<q\leq 2$.   } 
\label{miaocasefig3}
\end{figure}

\subsection*{Acknowledgements}

Jonathan Fraser was financially  supported by a \emph{Leverhulme Trust Research Fellowship} (RF-2016-500) and  an \emph{EPSRC Standard Grant} (EP/R015104/1).  Lawrence Lee was supported by an \emph{EPSRC Doctoral Training Grant} (EP/N509759/1). Ian Morris was supported by a \emph{Leverhulme Trust Research Project Grant} (RPG-2016-194). Han Yu was financially supported by the University of St Andrews.  The authors thank Kenneth Falconer for making several helpful comments on the paper.

\vspace{5mm}

Jonathan M. Fraser, School of Mathematics \& Statistics, University of St Andrews, St Andrews, KY16 9SS, UK
\textit{E-mail address}:\ \url{jmf32@st-andrews.ac.uk}

Lawrence D. Lee, School of Mathematics \& Statistics, University of St Andrews, St Andrews, KY16 9SS, UK
\textit{E-mail address}:\ \url{ldl@st-andrews.ac.uk}

Ian D. Morris, Mathematics Department, University of Surrey, Guildford, GU2 7XH, UK
\textit{E-mail address}:\ \url{i.morris@surrey.ac.uk}

Han Yu, School of Mathematics \& Statistics, University of St Andrews, St Andrews, KY16 9SS, UK
\textit{E-mail address}:\ \url{hy25@st-andrews.ac.uk}

\end{document}